\documentclass[11pt,reqno]{amsart}
\usepackage{amssymb}
\usepackage{amsthm}
\usepackage{amsmath}
\usepackage{amsthm}
\usepackage{amssymb}
\usepackage{amscd}
\usepackage{amsfonts}
\usepackage{array}
\usepackage{enumerate}
\usepackage{graphicx}
\usepackage[mathscr]{eucal}
\usepackage{latexsym}
\usepackage{epsfig}
\usepackage{graphics}
\usepackage{xcolor}
\usepackage[normalem]{ulem}
\usepackage{verbatim,amscd,amsmath,amsthm,amstext,amssymb,amsfonts,latexsym,lscape,rawfonts}

\usepackage{fullpage}

\newtheorem {theorem} {Theorem}
\newtheorem {definition}{Definition}
\newtheorem {proposition}{Proposition}
\newtheorem {corollary} {Corollary}
\newtheorem {lemma}  {Lemma}
\newtheorem {remark} {Remark}
\newtheorem {example}{Example}

\newcommand{\Q}{\mathbb{Q}}
\newcommand{\C}{\mathbb{C}}

\date{}

\title[Invariant algebraic surfaces]
{Invariant algebraic surfaces of polynomial vector fields in dimension three}

\author[N. Kruff, J. Llibre, C. Pantazi and S. Walcher]
{Niclas Kruff$^1$, Jaume Llibre$^2$,  Chara Pantazi$^3$\\ and
Sebastian Walcher$^4$}

\address{$^1$ $^4$ Lehrstuhl A f\"ur Mathematik, RWTH Aachen, 52056 Aachen,
Germany}\email{niclas.kruff@matha.rwth-aachen.de, walcher@mathA.rwth.aachen.de}

\address{$^2$ Departament de Matem\`{a}tiques, Universitat Aut\`{o}noma de Barcelona,
Edi\-fici C, 08193 Bellaterra, Barcelona, Catalonia, Spain}
\email{jllibre@mat.uab.cat}

\address{$^3$ Departament de Matem\`atiques, Universitat
Polit\`ecnica de Cata\-lunya, (EPSEB), Av. Doctor Mara\~{n}\'on, 44--50,
08028 Barcelona, Catalonia, Spain} \email{chara.pantazi@upc.edu}

\begin{document}

\begin{abstract}
We discuss criteria for the nonexistence, existence and computation of invariant algebraic surfaces for three-dimensional complex polynomial vector fields, thus transferring a classical problem of Poincar\'e from dimension two to dimension three. Such surfaces are zero sets of certain polynomials which we call semi--invariants of the vector fields. The main part of the work deals with finding degree bounds for irreducible semi--invariants of a given polynomial vector field that satisfies certain properties for its stationary points at infinity. As a related topic, we investigate existence criteria and properties for algebraic Jacobi multipliers. Some results are stated and proved for polynomial vector fields in arbitrary dimension and their invariant hypersurfaces. In dimension three we obtain detailed results on possible degree bounds. Moreover by an explicit construction we show for quadratic vector fields that the conditions involving the stationary points at infinity are generic but they do not a priori preclude the existence of invariant algebraic surfaces. In an appendix we prove a result on invariant lines of homogeneous polynomial vector fields.

{\bf AMS classification (2010):} 34A05, 34C45, 37F75, 17A60.

\end{abstract}

\maketitle

\section{Introduction}\label{intro}
Consider a polynomial ordinary differential equation in $\C^n$
\begin{equation}\label{first}
\dot{x}=f(x)=f^{(0)}(x)+f^{(1)}(x)+\cdots+f^{(m)}(x),
\end{equation}
with each $f^{(i)}$ a homogeneous polynomial of degree $i$, $0\leq i\leq m$, and $f^{(m)}\neq 0$. By $X_f$ we denote the vector field associated to $f$ (also called the Lie derivative with respect to $f$). A polynomial $\psi: \C^n\rightarrow \C^n$ is called a {\it semi--invariant} of $f$ if $\psi$ is nonconstant and
\begin{equation}\label{inv}
X_f(\psi)=\lambda\cdot \psi,
\end{equation}
for some polynomial $\lambda$, called the {\it cofactor} of $\psi$. As it is well--known, a polynomial is a semi--invariant of the vector field $f$ if and only if its vanishing set is invariant for the flow of system \eqref{first}. Given a degree bound, the problem of finding semi--invariants essentially reduces to solving a linear system of equations with parameters, thus a problem of linear algebra.

The existence problem for semi--invariants is relevant for several questions, notably  Darboux integrability and existence of Jacobi multipliers. From the work of \.Zo{\l}\c{a}dek \cite{Zoladek} it is known that generically no algebraic invariant sets exist for polynomial vector fields of a fixed degree; see also Coutinho and Pereira \cite{CouPer}. Even for dimension $n=2$ the existence problem is hard when the vector field has dicritical stationary points. When there are no dicritical stationary points then work of Cerveau and Lins Neto \cite{CerLN} and Carnicer \cite{Carn} provides degree bounds for irreducible semi--invariants; see Pereira \cite{Pere} for a refinement and also note the recent work by Ferragut et al. \cite{FGM}. For certain classes of planar polynomial vector fields it was shown in \cite{WPoinc} by elementary arguments that an effective degree bound for irreducible semi--invariants exists, and strong restrictions were found for possible integrating factors. For higher dimensions Jouanolou \cite{Jou} showed the existence of the general degree bound $m+1$ for semi--invariants of system \eqref{first} that define smooth hypersurfaces in projective space, {and Soares \cite{Soa97, Soa00,Soa01} extended and refined this result.} For dimension $n\geq3$ much work has been done in the last decade to classify and characterize invariant surfaces; see for instance the survey \cite{Cer2013} by Cerveau, and the work \cite{CerLNRV} by Cerveau et al. on local properties.

The purpose of the present paper, which is based in part on the doctoral thesis \cite{Kruffdiss} by one of the authors, is to generalize the {results of \cite{WPoinc} to higher dimensions, with a focus on dimension three. In contrast to the deep theoretical results used in most of the above mentioned references, our approach is different, employing rather elementary methods.} Thus we start with asking about the existence of irreducible, pairwise relatively prime semi--invariants
\begin{equation}\label{semi}
\phi_i=\phi_i^{(1)}+\cdots+\phi_i^{(r_i)}, \qquad 1\leq i\leq s,
\end{equation}
with $\phi_i^{(k)}$ homogeneous polynomials of degree $k$, and $\phi_i^{(r_i)}\neq 0.$ Our principal approach will be to consider stationary points at infinity of system \eqref{first}, and we obtain joint degree bounds for the $\phi_i$.

The plan of the paper is as follows. In a preparatory section \ref{preparation} we collect mostly known facts about semi--invariants of polynomial and formal vector fields, and Poincar\'e transforms in order to discuss the behavior at infinity. Some of the statements are proven for the reader's convenience. In section \ref{classvf} we consider a class of polynomial vector fields in $\mathbb C^n$ which is characterized by certain properties of its stationary points at infinity. We derive degree bounds for collections of irreducible semi--invariants given that either at least $n-1$ pairwise relatively prime semi--invariants exist or that a degree bound for (and a bound for the number of) the irreducible homogeneous semi--invariants of the highest degree term $f^{(m)}$ is known. We then proceed to discuss posssible exponents and degree bounds for Jacobi multipliers that are algebraic over the rational function field $\C(x_1,\ldots,x_n)$. We close the section by stating some facts about reduction of dimension for homogeneous vector fields. In section \ref{dimthreesec} we apply and specialize the results from the previous section to vector fields in $\mathbb C^3$, obtaining rather strong results on degree bounds by combining our approach with earlier results by Jouanolou \cite{Jou} and Carnicer \cite{Carn}. Moreover we explicitly construct a class of quadratic vector fields for which the conditions on the stationary points at infinity are directly verifiable. This class, seen as a subset of the coefficient space ($\cong \C^{18}$) contains a Zariski open subset, and the vector fields not satisfying the conditions on stationary points at infinity form a measure zero subset. The Appendix contains some additional material on the construction and some proofs. It also contains the statement and proof of a result by R\"ohrl \cite{Rohrlidpo} (the original source contains an erroneous statement and proof) which is needed by us as a basis for the construction of the quadratic vector fields in section \ref{dimthreesec}. R\"ohrl's result, which  generalizes the common knowledge fact that generically a linear map admits a basis of eigenvectors to homogeneous polynomial maps of arbitrary positive degree, seems to be of independent interest.
%%%%%%%%%%%%%%%%%%%%%%%%%%%%%%%%%%%%%%%%%%%%%%
\section{Preparations}\label{preparation}
In this section we review some known facts and tools to be used later on.

\subsection{Semi--invariants and some of their properties}
In addition to polynomial semi--invariants of polynomial vector fields we will consider local analytic (or formal) semi--invariants of local analytic (or formal) vector fields. Thus, a formal vector field is given by a power series
\begin{equation}\label{analfield}
g(x)=Bx+\sum_{i\geq 1} g^{(i)}(x) \qquad \mbox{in} \ \C^n,
\end{equation}
with $B$ linear and each $g^{(i)}$  a homogeneous polynomial of degree $i$. A {\it  semi--invariant} of $g$ is defined as a non-invertible power series
\begin{equation}\label{formalfield}
\rho=\rho^{(1)}+\rho^{(2)}+\cdots \quad  (\mbox{thus}\ \rho(0)=0)
\end{equation}
 satisfying $X_g(\rho)=\mu\rho$  for some power series $\mu.$ Similar to the polynomial case, an analytic function at $0$ is a semi--invariant of the analytic vector field \eqref{analfield} if and only if its vanishing set is invariant for $\dot{x}=g(x).$ We collect some general properties of semi--invariants.
\begin{lemma}\label{semilem}
\begin{enumerate}[(a)]
\item Let the polynomial vector field $f$ as in \eqref{first} be given. Then the following hold.
\begin{itemize}
\item From $X_f(\psi_j)=\lambda_j\cdot\psi_j$, $j=1,2$, the relations
\[
X_f(\psi_1\cdot\psi_2)=(\lambda_1+\lambda_2)\cdot\psi_1\cdot\psi_2\text{  and  }X_f(\psi_1/\psi_2)=(\lambda_1-\lambda_2)\cdot\psi_1/\psi_2
\]
follow.
\item If $\psi_1,\ldots,\psi_\ell$ are irreducible and pairwise relatively prime polynomials, $m_1,\ldots,m_\ell$ are nonzero integers and there is a polynomial $\mu$ such that
\[
X_f(\psi_1^{m_1}\cdots \psi_\ell^{m_\ell})=\mu\cdot \psi_1^{m_1}\cdots \psi_\ell^{m_\ell}
\]
then every $\psi_j$ is a semi-invariant of $f$.
\item If $\sigma$ is nonconstant and algebraic over the field $\mathbb C(x_1,\ldots,x_n)$ and satisfies $X_f(\sigma)=\mu\cdot\sigma$ for some polynomial $\mu$ then every nonzero coefficient $\beta$ of its minimal polynomial satisfies $X_f(\beta)=k\cdot\mu\cdot\beta$ with some positive integer $k$, hence also $X_f(\beta^{1/k})=\mu\cdot\beta^{1/k}$.
\end{itemize}
\item Mutatis mutandis, the same statements hold for formal semi-invariants of formal vector fields.
\end{enumerate}

\end{lemma}
\begin{proof}
The proofs for statement (b) are parallel to those for statement (a); so we only consider these. The first statement is straightforward, while the second is a direct consequence of the derivation property of $X_f$ and unique factorization in the polynomial ring. For the third statement, let
\[
T^m+\beta_1T^{m-1}+\cdots+\beta_m\in\C(x_1,\cdots,x_n)[T]
\] be the minimal polynomial of $\sigma$. Applying $X_f$ to
$\sigma^m+\beta_1\sigma^{m-1}+\cdots+\beta_m=0$ and using $X_f(\sigma)=\mu\cdot \sigma$, one gets
$$
0=\mu\cdot (m\sigma^m+\beta_1(m-1)\sigma^{m-1}+\cdots+\beta_{m-1}\sigma)+X_f(\beta_1)\sigma^{m-1}+\cdots+X_f(\beta_{m-1})\sigma+X_f(\beta_m).
$$
Subtract this relation from $m\cdot \mu\cdot(\sigma^m+\beta_1\sigma^{m-1}+\cdots+\beta_m)=0$ to obtain
$$
0=(\mu\cdot \beta_1-X_f(\beta_1))\sigma^{m-1}+\cdots+(m\cdot \mu\cdot \beta_m-X_f(\beta_m)).
$$
Since the minimal polynomial of $\sigma$ has degree $m$, all coefficients must vanish. Hence, for nonzero $\beta_k$ one finds $X_f(\beta_k)=k\cdot\mu\cdot \beta_k$.
\end{proof}
Given a semi--invariant $\psi$ of the polynomial vector field $f$, and a stationary point $z$ of $f$, one has either $\psi(z)\neq 0$
or $\psi$ is a local analytic (hence also a formal) semi--invariant of $f$ at $z.$ Thus one can use local information in the search for polynomial semi--invariants, based on the following result.
\begin{lemma}\label{polyloc}The following statements hold.
\begin{enumerate}[(a)]
\item Let $\psi$ be an irreducible polynomial semi-invariant of \eqref{first}, and $\psi(0)=0$. Then every irreducible series in the factorization of $\psi$ in the formal power series ring $\mathbb C[[x_1,\ldots,x_n]]$ has multiplicity one.
\item Let $\psi_1$ and $\psi_2$ be relatively prime polynomial semi-invariants of \eqref{first}, and $\psi_1(0)=\psi_2(0)=0$. Then the prime factorizations of $\psi_1$ and $\psi_2$ in the formal power series ring $\mathbb C[[x_1,\ldots,x_n]]$ have no common irreducible factor.
\end{enumerate}
\end{lemma}
These properties hold more generally for polynomials. An algebraic proof of this fact (which should be considered as known) is given in \cite{Kruffdiss}, Lemma 8.4, and we give an elementary ad-hoc proof in subsection \ref{polylocproof} of the Appendix.

We next discuss cases when the local information is rather precise.

\begin{lemma}\label{lconditions}
Let $g$ be a formal vector field as in \eqref{analfield}, with $\lambda_1,\cdots,\lambda_n$ the eigenvalues of $B$. Consider the following conditions:
\begin{enumerate}[1.]
\item $\lambda_1,\cdots,\lambda_n$ are linearly independent over the rational number field $\Q$.
\item $\dim_{\Q}\,(\Q\lambda_1+\cdots+\Q\lambda_n)=n-1$ and there exist positive integers $m_1,\ldots,m_n$ (w.l.o.g. relatively prime) with
$$
\sum_{i=1}^nm_i\lambda_i=0.
$$
\end{enumerate}
If one of the two conditions above is satisfied then the following hold.
\begin{enumerate}[(a)]
\item If $B=diag(\lambda_1,\ldots,\lambda_n)$ and $g$ is in Poincar\'e--Dulac normal form (PDNF)
then $x_1,\ldots,x_n$ are (up to multiplication by invertible series) the only irreducible formal semi--invariants of $g.$
\item For general $g$ there exist (up to multiplication by invertible series) precisely $n$ irreducible and pairwise relatively prime formal semi--invariants of system \eqref{analfield}.
\end{enumerate}
\end{lemma}
\begin{proof}
It suffices to prove statement (a), since a formal transformation to PDNF always exists. Thus let $g$ be in PDNF and let $\rho$ be an irreducible semi--invariant. According to Lemma 2.2 of  \cite{WPoinc}, up to multiplication with an invertible series one may assume that $X_B(\rho)=\alpha\rho$ for some constant $\alpha$, and $X_g(\rho)=\mu\rho$ with $X_B(\mu)=0$. \\If condition 1 holds then one has $g(x)=Bx$ (see e.g. Bibikov \cite{Bibi}, Theorem 2.1),  and for the series expansion
\[
\rho= \sum c_{d_1,\ldots,d_n}x_1^{d_1}\cdots x_n^{d_n}
\]
one obtains
\[
\alpha\rho =X_B(\rho)= \sum \left(\sum_i d_i\lambda_i\right)c_{d_1,\ldots,d_n}x_1^{d_1}\cdots x_n^{d_n},
\]
thus
\begin{equation}\label{rescond}
\alpha=\sum d_i\lambda_i\quad\text{ whenever  }c_{d_1,\ldots,d_n}\not=0.
\end{equation}
Since the $\lambda_i$ are linearly independent over $\Q$, one sees that only one coefficient $c_{u_1,\ldots,u_n}$ is nonzero, and therefore $\rho=x_1^{u_1}\cdots x_n^{u_n}$. The assertion about irreducible factors follows.\\
If condition 2 holds then one also arrives at \eqref{rescond}, but now, given distinct $(d_1,\cdots,d_n)$ and $(e_1,\cdots,e_n)$ with nonnegative integer entries such that
\[
\alpha=\sum d_i\lambda_i=\sum e_i\lambda_i,
\]
one has that $d_i-e_i=\ell\cdot m_i$ with some rational number $\ell$, due to the dimension assumption. Since the $m_i$ are relatively prime one sees that $\ell$ is an integer. Define $\gamma(x):=x_1^{m_1}\cdots x_n^{m_n}$, noting $X_B(\gamma)=0$. Let $(u_1,\cdots,u_n)$ be a nonzero vector with nonnegative integer entries such that $\sum m_i \cdot u_i=0$, and $\sum u_i$ minimal with respect to these properties. Then, repeated use of the argument above shows the existence of a nonnegative integer $j$ such that
\[
x_1^{d_1}\cdots x_n^{d_n}=\gamma(x)^j\cdot x_1^{u_1}\cdots x_n^{u_n}\quad\text{ whenever  }c_{d_1,\ldots,d_n}\not=0.
\]
This shows
\[
\rho=x_1^{u_1}\cdots x_n^{u_n}\cdot\sum \widehat c_j\gamma(x)^j
\]
and the only irreducible factors of $\rho$ are the $x_i$.
\end{proof}
Thus, one only has to consider finitely many local semi--invariants if condition 1 or 2 from  Lemma \ref{lconditions} holds.
Moreover,  their vanishing sets (in the analytic case) meet transversally in the stationary point $0$. This property is characteristic for conditions 1 and 2 above, as the next observation shows.
\begin{remark}
When the dimension is $n\geq 3$ and $\dim_{\Q}(\Q\lambda_1+\cdots+\Q\lambda_n)\leq n-2$, then there always exist infinitely many irreducible, pairwise relatively prime formal semi--invariants for the semisimple part $B_s$ of $B$: By the dimension assumption, there is a relation $\sum \ell_i\lambda_i=0$ with integers $\ell_i$ that do not all have the same sign;
thus we may assume that there are nonnegative, relatively prime integers $k_i$ and some $1<p<n$ such that
\[
\sum_{i=1}^p k_i\lambda_i=\sum_{i=p+1}^n k_i\lambda_i,
\]
whence $x_1^{k_1}\cdots x_p^{k_p}$ and $x_{p+1}^{k_{p+1}}\cdots x_n^{k_n}$ are semi-invariants with the same cofactor, and
\[
x_1^{k_1}\cdots x_p^{k_p} +\alpha\cdot  x_{p+1}^{k_{p+1}}\cdots x_n^{k_n}
\]
is a semi-invariant for every constant $\alpha$. Since the $k_i$ are relatively prime, irreducibility follows whenever $\alpha\not=0$.
\end{remark}
%%%%%%%%%%%%%%%%%%%%%%%%%%%%%%%%%%%%%%%%%%%%%%%
\subsection{Poincar\'e transforms and stationary points at infinity}
The following is geometrically motivated by the well known procedure of passing to the Poincar\'e hypersphere (or to projective space) in the analysis of polynomial vector fields. We choose a rather straightforward adaptation (see \cite{WPoinc} in the case of dimension two), to facilitate our computations.

\begin{definition}\label{pointranpol}
Let $\psi\in \C[x]$ be a polynomial of degree $r$ with decomposition
 \begin{equation*}
  \psi(x)=\sum\limits_{j=0}^{r}\psi^{(j)}(x),
 \end{equation*}
where each $\psi^{(j)}$ is a homogeneous polynomial of degree $j$ and $\psi^{(r)}\neq 0$.
\begin{enumerate}[(a)]
\item The homogenization\index{homogenization} of $\psi$ with respect to $x_{n+1}$ is defined as
\begin{equation*}
 \widetilde{\psi}(x_{1},\dots,x_{n},x_{n+1}):=\sum\limits_{j=0}^{r}\psi^{(j)}(x_{1},\dots,x_{n})x_{n+1}^{r-j}\in \C[x_{1},\dots,x_{n},x_{n+1}].\end{equation*}
\item The special {Poincar\'{e} transform} of $\psi$ with respect to
\begin{equation*}
e_1=(1,0,\ldots,0)^T\in \C^n
\end{equation*}
is the polynomial
\begin{equation*}\begin{split}
\psi^{*}:=\psi^{*}_{e_{1}}(x_{2},\dots,x_{n+1}):=&{\widetilde{\psi}}(1,x_{2},\dots,x_{n+1})\\
=&\sum\limits_{j=0}^{r}\psi^{(j)}(1,x_{2},\dots,x_{n})x_{n+1}^{r-j}\in \C[x_{2},\dots,x_{n},x_{n+1}].
\end{split}\end{equation*}
\item A {\em Poincar\'{e} transform} of $\psi$ w.r.t. $v\in \C^n\setminus \{0\}$ is defined as
 \[
  \psi^{*}_{v}(x_2,\dots,x_{n+1}):=\left(\psi\circ T^{-1}\right)^{*}_{e_1}(x_2,\dots,x_{n+1})
 \]
 with a regular matrix $T\in \C^{n\times n}$ such that
 $Tv=e_1$.
\end{enumerate}
\end{definition}

The following properties are easy to prove, see \cite{WPoinc} or \cite{Kruffdiss} for more details.
\begin{lemma}\label{lproperties}
Let $\psi$ and $v$ be as in Definition \ref{pointranpol}. Then the following hold.
\begin{itemize}
\item[(a)] One has $\psi^{(r)}(v)=0$ if and only if $\psi_v^{*}(0)=0.$
\item[(b)] The map
$$
\C[x_1,\cdots,x_n]\setminus\left<x_1\right>\rightarrow\C[x_2,\cdots,x_{n+1}], \quad \psi\mapsto \psi_{e_1}^*
$$
is injective.
\item[(c)] If $\psi$ is irreducible with $\psi^{(r)}(v)=0$ then $\psi_v^*$ is irreducible.

\item[(d)] If $\psi_1$ and $\psi_2$ are relatively prime and $\psi_1^*(v)=\psi_2^*(v)=0$, then $\psi_1^*$ and
$\psi_2^*$ are relatively prime.
\end{itemize}
\end{lemma}

Moreover, we define Poincar\'{e} transforms of vector fields.

\begin{definition}\label{pointranvec}
 Let $f$ be given as in \eqref{first}.
 \begin{enumerate}[(a)]
\item The homogenization of $f$ with respect to $x_{n+1}$ is defined as
\begin{equation*}
 \widetilde{f}(x_{1},\dots,x_{n+1}):=\begin{pmatrix}
       \sum\limits_{j=0}^{m}f^{(j)}(x_{1},\dots,x_{n})x_{n+1}^{m-j}\\
       0\\
      \end{pmatrix}=:
      \begin{pmatrix}
       g_{1}\\
       \vdots\\
       g_{n}\\
       0
      \end{pmatrix}\in \C[\underline{ x}]^{n+1},\text{  where  }\underline{x}=\begin{pmatrix}x_{1}\\ \vdots\\ x_{n+1}\end{pmatrix}.
\end{equation*}
\item The projection of $\widetilde{f}$ with respect to $x_{1}$ is
\begin{equation*}\begin{split}
 P\widetilde{f}(x_{1},\dots,x_{n+1}):=&-g_{1}(x_{1},\dots,x_{n+1})\cdot \underline{x}+x_{1}\cdot \widetilde{f}(x_{1},\dots,x_{n+1})\\
 =&\begin{pmatrix} 0\\ -g_{1}x_{2}+x_{1}g_{2}\\ \vdots\\ -g_{1}x_{n}+x_{1}g_{n}\\ -g_{1}x_{n+1}\\\end{pmatrix}.
\end{split}\end{equation*}

\item The special Poincar\'{e} transform with respect to the vector $e_1$ is defined as
\begin{equation*}\begin{split}\label{poinvec}
 f^{*}:=f^{*}_{e_{1}}(x_{2},\dots,x_{n},x_{n+1}):=\begin{pmatrix}-g_{1}(1,x_{2},\dots,x_{n+1})x_{2}+g_{2}(1,x_{2},\dots,x_{n+1})\\
 \vdots\\ -g_{1}(1,x_{2},\dots,x_{n+1}) x_{n}+g_{n}(1,x_{2},\dots,x_{n+1}) \\ -g_{1}(1,x_{2},\dots,x_{n+1})x_{n+1}\end{pmatrix},
\end{split}\end{equation*}
where $f^{*}\in \C[x_{2},\dots,x_{n},x_{n+1}]^{n}$.
\item A {\em Poincar\'{e} transform} of $f$ w.r.t. $v\in \C^n\setminus \{0\}$ is defined as
\[
  f^{*}_{v}(x_2,\dots,x_{n+1}):=\left(T\circ f\circ T^{-1}\right)^{*}_{e_1}(x_2,\dots,x_{n+1})
 \]
 with a regular matrix $T\in \C^{n\times n}$ such that $Tv=e_1$.
\end{enumerate}
\end{definition}

\begin{remark}\label{propertiesindependentofT}
 The definitions of Poincar\'{e} transforms depend on the choice of the matrix $T$, but a different choice of $T$ will just amount to a linear automorphism of the polynomial algebra resp. a conjugation of vector fields by a linear transformation; see \cite{WPoinc}, Lemma 3.1. This will be irrelevant for our purpose.
\end{remark}

We now turn to stationary points at infinity, i.e. to stationary points of Poincar\'e transforms which lie in the hyperplane $\{x:\,x_{n+1}=0\}$.
\begin{lemma}\label{l2}
Let $f$ be as in \eqref{first} and $v\in\C^n\setminus\{0\}.$ Then:
\begin{enumerate}[(a)]
\item The point $0$ is stationary for $f_v^*$ if and only if $f^{(m)}(v)=\gamma v$ for some $\gamma\in\C.$
\item The Jacobian $Df^{(m)}(v)$ admits the eigenvector $v$, with eigenvalue $m\gamma$. Let $\beta_2,\ldots,\beta_n$ be the further eigenvalues of $Df^{(m)}(v)$ (counted according to multiplicity). Then the eigenvalues of $Df_v^*(0)$ are given by
\[
-\gamma,\beta_2-\gamma,\ldots,\beta_n-\gamma.
\]

\item If the number of lines $\C v$ with $f^{(m)}(v)\in\C v$ is finite then it is equal to
$$
\frac{m^n-1}{m-1}=\sum_{i=0}^{n-1}m^i,
$$
counting multiplicities.
\end{enumerate}
\end{lemma}
\begin{proof}
Statement (a) can be verified directly. For statement (b) see  \cite{RohWal}, Proposition 1.8 and Corollary. 1.9. For statement (c), a proof is given by  R\"{ohrl} \cite{Rohrl}, using Bezout's theorem in projective space (see e.g. Shafarevich \cite{Sha}, Chapter 4  for this).
\end{proof}
\begin{remark} Note that  the stationary point $0$ of $f_v^*$ has multiplicity one if and only if all the eigenvalues of its Jacobian are nonzero. Moreover, the stationary point $0$ is then isolated.
In particular, given condition 1 or 2 of Lemma \ref{lconditions} for the eigenvalues at the stationary point $0$ of $f_v^*$, this stationary point has multiplicity one.
\end{remark}

%%%%%%%%%%%%%%%%%%%%%%%%%%%%%%%%%%%%%%%%%%%%%%%
\section{Invariant hypersurfaces of a class of vector fields}\label{classvf}
In the present section we will discuss a particular class of vector fields and degree bounds for irreducible semi--invariants of this class. We first give some definitions.
\begin{definition}\label{defproperties}
Let $f$ be given as in \eqref{first}.
\begin{itemize}
\item[(a)] We say that {\em $f$ has property $E$} if condition 1 or condition 2 in Lemma \ref{lconditions} holds for the linearization of $f_v^*$ at {\em every} stationary point at infinity.
\item[(b)] We say that {\em $f$ has property $S$} if every proper homogeneous invariant variety $Y$ of $f^{(m)}$ with $\dim Y\geq 1$ contains an invariant subspace $\C v$ with $v\neq 0.$
\end{itemize}
\end{definition}
Note that both conditions apply only to the highest degree term $f^{(m)}.$  In dimension two, property S is trivially satisfied, and it was shown in \cite{WPoinc}, Proposition 3.7 that property E is generic. In higher dimensions it is not a priori obvious that vector fields admitting properties E and  S exist. In section 4 we will construct a class of examples in dimension three.\\

One consequence of property E is that the number of stationary points at infinity is finite (otherwise some Jacobian at a stationary point would have to be non-invertible), and every stationary point at infinity has multiplicity one. A further consequence is not directly relevant for the main topic of this paper, but it is worth recording.
\begin{proposition}\label{curvprop}
Let $f$ satisfy property E. Then the number of irreducible invariant algebraic curves for system \eqref{first} is bounded by $(m^n-1)/(m-1)$.
\end{proposition}
\begin{proof} Let $C$ be an invariant algebraic curve of the system. Then its intersection with the hyperplane at infinity is invariant, hence every intersection point is stationary. Let $\mathbb Cv$ correspond to one of these points. Then, by Theorems 3.1 and 3.2 of
\cite{kruffinvid}, the local ideal defining the image of the curve under the Poincar\'e transform is generated by certain semi-invariants of $Df_v^*(0)$, and there must be $n-1$ of these, since the dimension of the curve equals one. Using property E and Lemma \ref{lconditions}, and the fact that the transformed curve is not contained in the hyperplane $\{x:\,x_{n+1}=0\}$, one sees that the only possible ideal is the one not containing the semi--invariant $x_{n+1}$. This argument shows that every stationary point at infinity can be an intersection point with at most one irreducible invariant curve. Now use Lemma \ref{l2}(c).
\end{proof}
%%%%%%%%%%%%%%%%%%%%%%%%%%%%%%%%%%%%%%%%%%%%%%%%%%%%%%%
\subsection{Degree bounds for semi--invariants}
We first state a preliminary result.
\begin{lemma}\label{splitlem}
Let $\phi_1,\ldots,\phi_d$ be irreducible and pairwise relatively prime semi--invariants of the polynomial vector field $f$, and let $v$ be such that $\phi_{i,v}^*(0)=0$, $1\leq i\leq d$, and that the Jacobian $Df_v^*(0)$ satisfies condition 1 or 2 from Lemma \ref{lconditions}.\\
Denote the irreducible local semi--invariants of $Df_v^*(0)$ by $\sigma_1,\ldots,\sigma_{n-1},\,\sigma_n:=x_{n+1}$.
Then $\sigma_n$ is not a factor of any $\phi_{i,v}^*$, and for the prime decompositions
\[
\phi_{i,v}^*=\sigma_1^{\ell_{i,1}}\cdots \sigma_{n-1}^{\ell_{i,n-1}}\cdot \nu_i,\quad 1\leq i\leq d
\]
with nonnegative integers $\ell_{i,j}$ and invertible series $\nu_i$ one has
\[
\sum_i \ell_{i,j}\leq 1.
\]
\end{lemma}
\begin{proof}
This is a direct consequence of Lemma \ref{polyloc}.
\end{proof}
Our first result yields degree bounds when $n-1$ irreducible semi--invariants exist.

\begin{theorem}\label{t1}
Let the vector field $f$ in \eqref{first} satisfy properties $E$ and $S$, and let $\phi_1,\ldots, \phi_{n-1}$ be irreducible and pairwise relatively prime semi--invariants of $f$ as in \eqref{inv}, of degrees $r_1,\ldots,r_{n-1}$ respectively. Then
$$
\prod_{i=1}^{n-1}r_i\leq \frac{m^n-1}{m-1}.
$$
\end{theorem}
\begin{proof} We first recall that invariance of the common zero set of the $\phi_j$ for \eqref{first} implies invariance of  $\mathcal{V}(\phi_1^{(r_1)},\ldots, \phi_{n-1}^{(r_{n-1})})$ for $\dot x=f^{(m)}(x)$.
Now let $\C v$ be a common zero of $\phi_1^{(r_1)},\ldots,\phi_{n-1}^{(r_{n-1})}.$ By property $S$ we may assume that $f^{(m)}(v)\in\C v$. Consider the prime factorization of the $\phi_{i,v}^*$ in the power series ring. By property $E$ and Lemma \ref{splitlem}, after possibly renumbering the factors, we have
\[
\phi_{i,v}^*=\sigma_i\cdot \nu_i,\quad 1\leq i\leq d,
\]
whence the intersection multiplicity of the common zero of the $\phi_{i,v}^*$ and $x_{n+1}$ is equal to one. In particular the irreducible component $Y$ of
 $\mathcal{V}(\phi_1^{(r_1)},\ldots, \phi_{n-1}^{(r_{n-1})})$ that contains $\C v$ is equal to $\C v$.
 Bezout's Theorem in projective space now shows that
 $\mathcal{V}(\phi_1^{(r_1)},\ldots, \phi_{n-1}^{(r_{n-1})})$ is the union of precisely $\prod_{i=1}^{n-1}r_i$ distinct lines in $\C^n$. Each of these lines is an invariant set for $\dot{x}=f^{(m)}(x),$ hence by Lemma \ref{l2}(c) one has the assertion.
\end{proof}
\begin{corollary}
Let $f$ be as in \eqref{first}, satisfying properties $E$ and $S$. Moreover, let $k\geq n$ and $\phi_1,\ldots,\phi_k$ be irreducible and pairwise relatively prime semi--invariants of $f$, with $deg\ \phi_i=r_i$. Then,
$$
\prod_{1\leq i\leq k}r_i\cdot \sum_{\begin{array}{c}M\subseteq \{1,\ldots,k\}\\|M|=k+1-n\end{array}}\frac1{\prod_{j\in M}r_j} \leq \frac{m^n-1}{m-1}.
$$
\end{corollary}
\begin{proof}
Let $1\leq i_1<\cdots <i_{n-1}\leq k$, thus $\{i_1,\ldots,i_{n-1}\}$ is the complement of a set $M\subseteq\{1,\ldots,k\}$ with $k+1-n$ elements. Then, by the argument in the proof of Theorem 1, the vanishing set
$$
\mathcal{V}\left(\phi_{i_1}^{(r_{i_1})},\ldots,\phi_{i_{n-1}}^{(r_{i_{n-1}})}\right)
$$
is a union of precisely $r_{i_1}\cdots r_{i_{n-1}}$ invariant lines, each with multiplicity one. By property E and Lemma \ref{lconditions}(b), each invariant line $\C v$ of $f^{(m)}$ appears in at most one of the common vanishing sets
$$
\mathcal{V}\left(\phi_{i_1}^{(r_{i_1})},\ldots,\phi_{i_{n-1}}^{(r_{i_{n-1}})}\right),
$$
since otherwise more than $n$ local invariant hypersurfaces would meet at the stationary point $0$ of $f_v^*$.
Now Bezout's theorem, and adding up the contributions of all stationary points at infinity, shows the assertion.
\end{proof}

The second result may be used to obtain degree bounds for semi--invariants of $f$, if degree bounds for semi--invariants of the highest degree term $f^{(m)}$ are known.
\begin{theorem}\label{t2}
Let the vector field $f$ in \eqref{first} satisfy properties $E$ and  $S$, and let $\phi_1,\ldots,\phi_s$ be irreducible and pairwise relatively prime semi--invariants of $f$.
\begin{itemize}
\item[(a)] Let $\{\psi_1,\ldots,\psi_\ell\}$ be the set of pairwise relatively prime irreducible factors of the $\phi_i^{(r_i)}$ with $1\leq i\leq s$. (Note that these are homogeneous and semi--invariants of $f^{(m)}$.) Then, given the representations
    $$
    \phi_i^{(r_i)}=\prod_{j=1}^\ell \psi_j^{k_{ij}}, \quad k_{ij}\geq 0,
    $$
one has $\sum_{i=1}^sk_{ij}\leq 1$ for all $j$. Thus every $\psi_j$ appears in at most one of the $\phi_i^{(r_i)}$, and if it appears then with exponent 1.

\item[(b)] If $f^{(m)}$ admits only finitely many irreducible semi--invariants  $\psi_1,\ldots,\psi_\ell$ (up to multiplication by constants), then
$$
deg \ \phi_i=\sum_{j=1}^\ell k_{ij}\ deg\  \psi_j, \quad 1\leq i\leq s
$$
and
$$
\sum_{i=1}^s deg\  \phi_i\leq \sum_{j=1}^\ell deg\  \psi_j.
$$
In particular the number of irreducible and pairwise relatively prime semi--invariants of $f$ is finite.
\end{itemize}
\end{theorem}
\begin{proof} Statement (b) is a simple consequence of (a).
As for the proof of statement (a), by property S we may again consider the $\phi_i^{(r_i)}$ at a common zero $\C v$ which is also invariant for
$\dot{x}=f^{(m)}(x).$ We consider a Poincar\'e transform at $v$, and there is no loss of generality in assuming that $v=e_1$. Moreover let the $\sigma_i$ be as in Lemma \ref{splitlem}, and by property E and Lemma \ref{lconditions} we may assume that
\[
\sigma_i(x_2,\ldots,x_n,\,x_{n+1})=x_{i+1}+\text{ h.o.t},\quad 1\leq i\leq n-1,
\]
with ``h.o.t'' denoting higher order terms. From Definition \ref{pointranpol} we have
\[
\phi_i^{(r_i)}(1,x_2,\ldots,x_n)+x_{n+1}\cdot(\cdots)=\phi_{i,e_1}^*=\sigma_1^{\ell_{i,1}}\cdots \sigma_{n-1}^{\ell_{i,n-1}}\cdot \nu_i
\]
with all $\ell_{i,j}\in \{0,1\}$ and therefore also
\[
\phi_i^{(r_i)}(1,x_2,\ldots,x_n)=\sigma_1(x_2,\ldots,x_n,0)^{\ell_{i,1}}\cdots \sigma_{n-1}(x_2,\ldots,x_n,0)^{\ell_{i,n-1}}\cdot \nu_i(x_2,\ldots,x_n,0).
\]
From $\sigma_i=x_{i+1}+\cdots$ one sees that the $\sigma_i(x_2,\ldots,x_n,0)$ are still irreducible and pairwise relatively prime in the formal power series ring. Comparing this decomposition with
\[
  \phi_i^{(r_i)}=\prod_{j=1}^\ell \psi_j^{k_{ij}},
\]
the corresponding local decompositions of the $\psi_j^{k_{ij}}(1,\,x_2,\ldots,x_n)$ can have only simple prime factors, hence  $k_{ij}\leq 1$ whenever $\psi_j(v)=0$.
\end{proof}

%%%%%%%%%%%%%%%%%%%%%%%%%%%%%%%%%%%%%%%%%%%%%%%%%%%
\subsection{Degree bounds for Jacobi multipliers}
We recall that a {\em Jacobi multiplier} (or Jacobi last multiplier) $\sigma\neq 0$ of a vector field $h$ is characterized by the condition
$$
X_h(\sigma)+{\rm div}\, h\cdot \sigma=0,
\text{ equivalently    } X_h(\sigma^{-1})={\rm div}\, h\cdot \sigma^{-1}.$$
In particular, a Jacobi multiplier satisfies the defining identity for semi--invariants but might be invertible. For a comprehensive account of Jacobi multipliers and their properties see Berrone and Giacomini \cite{BerGia}. \\We focus on Jacobi multipliers that are algebraic over $\C(x_1,\ldots,x_n)$ resp. $\C((x_1,\ldots,x_n))$. This is a natural requirement in dimension $n=2$ in view of Prelle and Singer's paper \cite{PreSin} on elementary integrability, and seems to be a sensible ( initial) restriction also for local integrating actors of Darboux type, as the following results will imply.
We note that by Lemma \ref{semilem} one may restrict interest to products of powers of semi--invariants, with rational exponents. We first discuss the local analytic, respectively the formal case, given the setting of Lemma \ref{lconditions}.

\begin{lemma}\label{ljac}
Let $g(x)=Bx+\cdots$ as in \eqref{analfield} in Poincar\'e-Dulac normal form, $B={\rm diag}\,(\lambda_1,\ldots,\lambda_n)$.
\begin{itemize}
\item[(a)] If condition 1 from Lemma \ref{lconditions} holds, then $(x_1\cdots x_n)^{-1}$ is, up to multiplication by constants, the only Jacobi multiplier of $g$ which is algebraic over $\C((x_1,\ldots,x_n)).$
\item[(b)] If condition 2 from Lemma \ref{lconditions} holds, and $\gamma=x_1^{m_1}\cdots x_n^{m_n}$, then there exist linearly independent diagonal matrices $C_1,\ldots,C_{n-1}$ such that the Lie bracket $[g, C_ix]=0$ and $X_{C_ix}(\gamma)=0$, $1\leq i\leq n-1$; moreover $B$ is a linear combination of the $C_i$.
    \begin{itemize}
    \item If $\tau(x):=det(g(x), C_1x,\ldots,C_{n-1}x)\neq 0$, then there exists $\ell>0$ such that $\tau(x)=\gamma(x)^\ell\cdot (x_1\cdots x_n)\cdot \widehat{\tau}(x),$ with an invertible series $\widehat{\tau}$, and $\tau^{-1}$ is (up to multiplication by constants) the unique Jacobi multiplier for $g$ which is algebraic over $\C((x_1,\ldots,x_n)).$

    \item If $\tau(x)=0$ then $(x_1\cdots x_n)^{-1}$ is a Jacobi multiplier, and $\gamma$ a first integral of $g$. In that case every Jacobi multiplier that is algebraic over $\C((x_1,\ldots,x_n))$ has the form
\[
(x_1\cdots x_n)^{-1}\cdot \nu
\]
with $d\in\Q$ and $\nu$ an algebraic first integral of $g$.
    \end{itemize}
\end{itemize}
\end{lemma}
\begin{proof}

\noindent (a) In this case one has $g(x)=Bx$, and one directly verifies that $\sigma:= (x_1\cdots x_n)^{-1}$ is a Jacobi multiplier. Assume that there exists a Jacobi multiplier $\widetilde\sigma$ that is algebraic over $\C((x_1,\ldots,x_n))$, then one also has a Jacobi multiplier $x_1^{-d_1}\cdots x_n^{-d_n}$ with rational exponents by Lemma \ref{semilem}, and
\[
\sum d_i\lambda_i=\sum \lambda_i.
\]
Condition 1 implies that all $d_i=1$, hence $x_1^{-1}\cdots x_n^{-1}$ is the unique algebraic integrating factor. Using the (notation and) argument in the proof of Lemma \ref{semilem} the coefficient of $T^{m-i}$ in the minimal polynomial of $\widetilde\sigma^{-1}$ is a constant multiple of $x_1\cdots x_n$. But then the same holds for $\widetilde\sigma^{-1}$ itself.

\noindent (b) It is known that
$$
g(x)=Bx+\sum_{j\geq 1}\gamma(x)^jD_jx
$$
with diagonal matrices $D_j$, see e.g. Bibikov \cite{Bibi}, Definition 2.3 and Theorem 2.2. The equation $\sum m_i\mu_i=0$ has $n-1$ linearly independent solutions in $\Q^n$.
Take these as the diagonal elements of $C_1,\cdots,C_{n-1}$ respectively. Then $X_{C_i}(\gamma)=0$ and $[C_i, B]=[C_i, D_j]=0$ for all
$i,j$, whence $[C_i, g]=0$ for all $i$. Now write $D_jx=\sum \nu_{jh} C_hx+\beta_j\cdot x$ with constants $\nu_{jh}$ and $\beta_j$. Thus,
$$
\tau(x)=\sum_{j\geq 1} \gamma(x)^j\beta_j \det(x, C_1x,\ldots, C_{n-1}x)=\kappa x_1\cdots x_n\sum_{j\geq 1}\beta_j\gamma(x)^j,
$$
with a nonzero constant $\kappa$.
If $\tau\neq 0$ then $\tau^{-1}$ is a Jacobi multiplier according to Berrone and Giacomini \cite{BerGia}, and if $\ell$ is the smallest index with $\beta_\ell\neq 0$ then we get
$$
\tau(x)=\gamma^\ell(x)\cdot(x_1\cdots x_n)\cdot (\beta_\ell \kappa+h.o.t).
$$
As for uniqueness, we first recall: Every first integral in $\C((x_1,\ldots,x_n))$ of $\dot x=Bx$ is a quotient of power series in $\gamma$. Indeed, numerator and denominator are semi-invariants of $B$, with the same cofactor (which is a first integral of $B$ and lies in $\C[[x_1,\ldots,x_n]]$), and the argument in the proof of Lemma \ref{lconditions} shows the assertion.\\
From this we find that $g$ admits no nonconstant first integral that is algebraic over $\C((x_1,\ldots,x_n))$ whenever some $\beta_\ell\neq 0$, by showing that there exists no such first integral in $\C((x_1,\ldots,x_n))$. But such a first integral would also be a first integral of $\dot x=Bx$ (see e.g. \cite{LWZ}, Theorem 1), hence a quotient of power series in $\gamma$. But then $\gamma$ is a first integral.
Finally, the identity
$$
X_g(\gamma)=\kappa\cdot\sum m_i\cdot \sum_{j\geq 1} \beta_j\gamma(x)^{j+1}
$$
yields a contradiction unless all $\beta_j=0$.\\ In the case $\tau=0$ one verifies by direct computation that $(x_1\cdots x_n)^{-1}$ is a Jacobi multiplier, and $\gamma$ is obviously a first integral of $g$. The last statement is again clear from known properties of Jacobi multipliers; see \cite{BerGia}.
\end{proof}

Obviously, Lemma \ref{ljac} is also applicable to local analytic or formal systems which are not in PDNF, given that condition 1 or condition 2 of Lemma \ref{lconditions} holds. We shall now apply this to Jacobi multipliers of system \eqref{first} that are algebraic over the rational function field $\C(x_1,\ldots,x_n)$.

\begin{theorem}\label{tjm}
Let the polynomial vector field $f$ given by \eqref{first} satisfiy properties $E$ and $S$ and let
$$
\left(\phi_1^{d_1}\cdots \phi_s^{d_s}\right)^{-1}
$$
be a Jacobi multiplier of $f$, with the $\phi_i$ as in \eqref{semi}, and nonzero rational exponents $d_1,\cdots,d_s$.

Moreover, assume that there exists a line $\C w$ such that linearization of $f_w^*$ at 0 has eigenvalues that  are linearly independent over $\Q$. Then
$$
d_1=\cdots=d_s=1 \quad \mbox{and}\quad \sum_{i=1}^sr_i=m+n-1.
$$
\end{theorem}
\begin{proof}
(i) We need some technical preparations, the proofs of which are straightforward generalizations of the ones for Proposition 3.3 in \cite{WPoinc}.
\begin{itemize}
\item If $\psi$ is a semi--invariant of $f$, with degree $r$ and $X_f(\psi)=\lambda \psi$, then the Poincar\'e transform $\psi_{e_1}^*$ is a semi--invariant of $f^*_{e_1}$ with cofactor $-rg_1(1,x_2,\ldots,x_{n+1})+\lambda_{e_1}^*.$
\item $f_{e_1}^*$ admits the Jacobi multiplier
$$
\left( x_{n+1}^{(m+n-\sum d_ir_i)}\cdot (\phi_{1,e_1}^*)^{d_1}\cdots (\phi_{s,e_1}^*)^{d_s} \right)^{-1}.
$$\item More generally, for any $v\in \C^n\setminus\{0\}$, the Poincar\'e transform $f_v^*$  admits the Jacobi multiplier
$$
\left( x_{n+1}^{(m+n-\sum d_ir_i)}\cdot (\phi_{1,v}^*)^{d_1}\cdots (\phi_{s,v}^*)^{d_s} \right)^{-1}.
$$
\end{itemize}
(ii) Let $\C w$ correspond to a stationary point at infinity such that the eigenvalues of $Df_w^*(0)$ are linearly independent over
$\Q$. Then Lemma \ref{ljac} (a) shows that $d_i=1$ for all $i$ such that $\phi_{i,w}^*(0)=0$, and moreover $m+n-\sum d_ir_i=1$.

(iii) If $\C v$ corresponds to a stationary point at infinity such that the eigenvalues of  $Df_v^*(0)$ satisfy the second condition of Lemma \ref{lconditions}, then one exponent in the local factorization of the multiplier, viz. the one belonging to the factor $x_{n+1}$, is equal to $-1$. By Lemma \ref{ljac}(b), all the other exponents equal $-1$, hence $d_i=1$ whenever $\phi_{i,v}^*(0)=0$. Due to property $S$, we thus find that all exponents are equal to $-1$, and the assertion follows.
\end{proof}

%%%%%%%%%%%%%%%%%%%%%%%%%%%%%%%%%%%%%%%%%%%%%%%%
\subsection{Reduction of dimension}
The verification of property S in dimension three, as well as the search for degree bounds via Theorem \ref{t2} leads to semi--invariants of the homogeneous vector field $f^{(m)}$, and in turn to semi--invariants of a reduced vector field in $\C^{n-1}$. This reduction is due to scaling symmetry, and we will recall it now.

\begin{proposition}\label{redprop}
Let $p:\C^n\rightarrow\C^n$, $x\mapsto (p_1(x), \ldots, p_n(x))^T$ be homogeneous of degree $m$, and let $H=\left\{x:\,x_n=0\right\}.$
\begin{enumerate}[(a)]
\item Then
$$
\Phi: \C^n\setminus H\rightarrow\C^{n-1}, \quad x\mapsto\left(\begin{array}{c}x_1/x_n\\ \vdots \\ x_{n-1}/x_n\end{array}\right)
$$
maps solutions orbits of $\dot{x}=p(x)$ to solution orbits of $\dot{y}=q(y)$, with
$$
q=\left(\begin{array}{c}
q_1\\ \vdots \\ q_{n-1}
\end{array}
\right); \quad q_i(y)=p_i(y_1,\cdots,y_{n-1}, 1)-y_ip_n(y_1,\cdots,y_{n-1},1).
$$
\item Every homogeneous invariant set $Y$ of $\dot{x}=p(x)$ which is not contained in $H$ is mapped to an invariant set of $\dot{y}=q(y)$ with dimension decreasing by one. Conversely, the inverse image of every invariant set $Z$ of
    $\dot{y}=q(y)$ is a homogeneous invariant set of $\dot{x}=p(x),$ with dimension increasing by one.
\item Let $v\in \C^n$ such that $v_n\not=0$ and $p(v)=\gamma v$. Then $v$ is an eigenvector of $Dp(v)$ with eigenvalue $m\gamma$. Let $\beta_2,\ldots,\beta_{n}$ be the other eigenvalues of $Dp(v)$, each counted according to its multiplicity. Then the linearization of $q$ at the stationary point $(v_1/v_n,\ldots,v_{n-1}/v_n)$ has eigenvalues
\[
v_n(\beta_2-\gamma),\ldots,v_n(\beta_{n}-\gamma).
\]
\end{enumerate}
\end{proposition}
\begin{proof}[Sketch of proof]
From $\dot{x}_i=p_i(x)$ one gets
$$
\frac{d}{dt}\left(
\frac{x_i}{x_n}\right)=\frac{1}{x_n}(x_{n}p_i(x)-x_ip_n(x)), \quad 1\leq i\leq n-1
$$
and rescaling time yields
$$
\left(
\frac{x_i}{x_n}\right)'=x_np_i(x)-x_ip_n(x), \quad 1\leq i\leq n-1.
$$
Dehomogenize to obtain statement (a).

Statement (b) is a direct consequence of (a) since any invariant set is a union of solution orbits.

The proof of statement (c) is a variant of the proof of \cite{RohWal}, Proposition 1.8:  Define
\[
Q(x):=x_n\cdot p(x)-p_n(x)\cdot x,\text{   with   } DQ(x)y= y_n\cdot p(x)+x_n\cdot Dp(x)y-(Dp(x)y)_n\cdot x -p_n(x)y
\]
and note that the first $n-1$ entries of $Q(x)$, upon setting $x_n=1$, are just the entries of $q(x)$. Now let $y$ be an eigenvector of $Dp(v)$ that is linearly independent from $v$, with eigenvalue $\beta$. Then, using $p(v)=\gamma v$, one gets
\[
DQ(v)y= \left(\cdots\right)\cdot v+ v_n\cdot Dp(v)y-\gamma v_n\cdot y=\left(\cdots\right)\cdot v+v_n\cdot(\beta-\gamma)\cdot y.
\]
This proves the part of the assertion for eigenvectors. The remaining part (if nontrivial Jordan blocks exist) is proven similarly.
\end{proof}
\begin{remark}\label{redrem}
\begin{enumerate}[(a)]
\item Note that the entries of $q$ are just the first $n-1$ entries of the Poincar\'e transform of $p$ with respect to $e_n$.
\item A coordinate--free version of the reduction starts from a nonzero linear form
\[
\alpha(x)=\sum\alpha_ix_i, \quad H_\alpha:=\{x:\,\alpha(x)=0\}\text{  and  } \Psi_{\alpha}: \C^n\setminus H_\alpha\to\C^n,\,\,x\mapsto\frac1{\alpha(x)}\,x.
\]
Then $\Psi_\alpha$ maps solution orbits of $\dot x=p(x)$ to solution orbits of the equation
\[
\dot x=Q_\alpha(x):=\alpha(x)p(x)-\alpha(p(x))x
\]
which admits the linear first integral $\alpha$.\\
In this case, whenever $p(v)=\gamma v$, $\alpha(v)\not=0$ and the eigenvalues of $Dp(v)$ are as in Proposition \ref{redprop} , the eigenvalues for the reduced system on the hyperplane given by $\alpha(x)=1$ are
\[
\alpha(v)\cdot(\beta_2-\gamma),\ldots,\alpha(v)\cdot(\beta_{n}-\gamma).
\]
\end{enumerate}
\end{remark}
%%%%%%%%%%%%%%%%%%%%%%%%%%%%%%%%%%%%%%%
\section{Dimension three}\label{dimthreesec}
In this section we will specialize our general results to dimension $n=3$. In particular we will verify that property $S$ from Definition \ref{defproperties} is always satisfied, and show that property E holds for almost all quadratic vector fields (in a sense to be specified).
\subsection{Property S and reduction}The first pertinent property is always satisfied in dimension three, as follows directly from the work of Jouanolou \cite{Jou}.
\begin{proposition}\label{condsprop} Let $f$ be a polynomial vector field in $\mathbb C^3$. Then $f$ satisfies property S.

\end{proposition}
\begin{proof} In dimension three one has to prove that the zero set of a homogeneus semi--invariant of a homogeneous polynomial vector field $p$ contains an invariant line for $p$. There is no loss of generality in assuming that the entries of $p$ are relatively prime. One may rephrase this for the projective plane $\mathbb P^2(\C)$, by introducing the one form
\[
\omega=p_1\,{\rm d}x_1+p_2\,{\rm d}x_2+p_3\,{\rm d}x_3
\]
and considering a homogeneous polynomial solution $\phi$ of the Pfaffian equation $\omega=0$. First consider the case when the projective curve defined by the zeros of $\phi$ is smooth (hence normal). Then by Jouanolou, Chapter 2, Proposition 4.1(ii), the curve defined by $\phi$ in $\mathbb P^2$ contains a singular point of $\omega$, which corresponds to an invariant line for the homogeneous vector field $p$. If the normality requirement for the solution is not satisfied then the projective curve defined by $\phi$ contains a singular point, which translates to a singular line in homogeneous coordinates. But singular sets of invariant varieties of polynomial vector fields are also invariant.
\end{proof}
We note that from Jouanolou \cite{Jou}, Chapter 2, Proposition 4.1(iii) one also obtains the degree bound $m+1$ for irreducible homogeneous semi--invariants whose associated projective curve is smooth.\\
Next we discuss the reduction of the highest degree term of $f$, with a view on applying Theorem \ref{t2}.
\begin{lemma}\label{carnlem}
Let the polynomial vector field $f$ be given as in \eqref{first}, and consider the reduction of its homogeneous highest degree term $p=f^{(m)}$, according to Proposition \ref{redprop} and Remark \ref{redrem}. Then the following hold.
\begin{enumerate}[(a)]
\item  Upon identifying homogeneous polynomial vector fields with their coefficients in some $\C^N$, the reduction $q$ of $p$ admits no stationary points at infinity for a Zariski--open subset of $\C^N$.
\item Assume $p(v)=v$ and let the linear form $\alpha$ be such that $\alpha(v)\not=0$. Moreover let the eigenvalues of $Dp(v)$ be $m$, $\beta_2$ and $\beta_3$. Then the eigenvalues of the linearization of $q$ at the corresponding stationary point are $\alpha(v)(\beta_2-1)$ and $\alpha(v)(\beta_3-1)$.
\item Assume that $p(v)=v$ and  the linearization of $f_v^*$ at the stationary point at infinity of system \eqref{first} satisfies condition 1 or condition 2 from Lemma \ref{lconditions}. Then the $\beta_i-1$ are nonzero and their ratio is not a positive rational number.
\end{enumerate}
\end{lemma}
\begin{proof} We may assume that $\alpha(x)=x_3$.
For statement (a), abbreviate $h_i(y_1,y_2):=p_i(y_1,y_2,1)$, noting that the $h_i$ generically have degree $m$. We now compute the Poincar\'e transform of $q$ with respect to $e_1$, following the procedure in Definition \ref{pointranvec}. The first two entries of the homogenization have the form
\[
-h_3^{(m)}(y_1,y_2)\cdot\begin{pmatrix}y_1\\y_2\end{pmatrix}+y_3\cdot\left(\begin{pmatrix}h_1^{(m)}(y_1,y_2)\\h_2^{(m)}(y_1,y_2)\end{pmatrix}-h_3^{(m-1)}(y_1,y_2)\cdot\begin{pmatrix}y_1\\y_2\end{pmatrix}\right)+y_3^2\cdots,
\]
from which the Poincar\'e transform is computed as
\[
q_{e_1}^*(y_2,y_3)=y_3\cdot\begin{pmatrix}-h_1^{(m)}(1,y_2)\cdot y_2+h_2^{(m)}(1,y_2)\\-h_3^{(m)}(1,y_2)\end{pmatrix}+y_3^2\cdots
\]
Reducing this vector field by dividing out the factor $y_3$ yields a vector field that generically (i.e. corresponding to a Zariski open set in the space of coefficients of the $h_i$, hence also of the coefficients of $p$) has no stationary points on $y_3=0$, since $-h_1^{(m)}(1,y_2)\cdot y_2+h_2^{(m)}(1,y_2)$ and $h_3^{(m)}(1,y_2)$ generically have no common zeros.\\
Statement (b) is a direct consequence of Proposition \ref{redprop}. To prove statement (c), note first that conditions 1 and 2 both imply that all eigenvalues for the Poincar\'e transform are nonzero and observe Lemma \ref{l2}. Now assume that $(\beta_3-1)/(\beta_2-1)=r/s$ with positive integers $r$ and $s$. Then one obtains
\[
(r-s)\cdot(-1)+r\cdot\beta_2+(-s)\cdot\beta_3=0.
\]
Therefore the eigenvalues of $Df_v^*(0)$ are linearly dependent over $\mathbb Q$, hence condition 1 cannot hold. Moreover, condition 2 also cannot hold because the integer coefficients in the linear combination have different signs.
\end{proof}
Now we are ready to determine degree bounds, applying a result of Carnicer.
\begin{theorem}
Let the polynomial vector field $f$ of degree $m$ be given on $\C^3$, and assume that
\begin{enumerate}[(i)]
\item the reduction of the homogeneous highest degree term $f^{(m)}$ admits no stationary points at infinity;
\item the vector field has property E.
\end{enumerate}
Then the following hold.
\begin{enumerate}[(a)]
\item Every irreducible homogeneous semi--invariant of $f^{(m)}$ has degree $\leq m+1$.
\item There exist (up to scalar multiples) only finitely many irreducible homogeneous semi--invariants $\psi_1\ldots,\psi_\ell$  of $f^{(m)}$, and in case $\ell\geq2$ one has
\[
\sum_{1\leq i<j\leq \ell}{\rm deg}\,\psi_i\cdot {\rm deg}\psi_j\leq(m^n-1)/(m-1);
\]
in particular $\ell(\ell-1)/2\leq(m^n-1)/(m-1)$.
\item The vector field $f$ admits only finitely many irreducible and pairwise relatively prime semi--invariants.
\end{enumerate}
\end{theorem}
\begin{proof}
We first prove statement (a). The condition in Carnicer's theorem \cite{Carn} is that no singular point of the corresponding one-form in the projective plane is dicritical. Given a non-nilpotent Jacobian, dicritical singular points are characterized by positive rational eigenvalue ratio. But all the singular points of the one-form correspond to stationary points of the reduction of $f^{(m)}$ in the affine plane, thanks to condition (i), and at every stationary point the linearization is invertible and the eigenvalue ratio is not a positive rational number, by property E and Lemma \ref{carnlem}. Thus Carnicer's theorem yields the degree bound. (Note that the degree of the foliation is less than or equal to one plus the degree of $q$.)

For statement (b), let $\psi_1$ and $\psi_2$ be irreducible and relatively prime semi-invariants. Then by Bezout's theorem they intersect in ${\rm deg}\psi_1\cdot{\rm deg}\psi_2$ points. Property E ensures that every intersection point is of multiplicity one, and none of these intersection points is contained in the vanishing set of another irreducible semi-invariant (observe Lemma \ref{polyloc} and note that  locally there are just two invariant curves passing through each singular point; see e.g. \cite{WPoinc}, Theorem 2.3). Now add up the contributions of pairs of semi--invariants and use Lemma \ref{l2}.

The assertion of statement (c) follows readily with Proposition \ref{condsprop} and Theorem \ref{t2}.
\end{proof}
%%%%%%%%%%%%%%%%%%%%%%%%%%%%%%%%%%%%%%%%%%%%%%
\subsection{Quadratic vector fields in dimension three}
We still have to show that vector fields with property E actually exist in dimension 3. Since only the homogeneous highest degree terms are involved in these conditions one may restrict attention to homogeneous polynomial vector fields (and add arbitrary terms of smaller degree). A direct verification for a given homogeneous vector field is problematic, because determining the invariant lines explicitly (which would seem a natural first step in a straightforward approach) is generally not possible. Therefore we take a roundabout approach, explicitly constructing vector fields by prescribing invariant lines.

In this subsection we will show that property E is generically satisfied for degree two vector fields in $\mathbb C^3$. It suffices to
consider homogeneous quadratic maps
\begin{equation}\label{quadcoeff}
p:\,\mathbb C^3\to\mathbb C^3; \quad p(x)=\left(\sum_{i, j: i<j}\beta_{i,j,k}x_ix_j\right)_{1\leq k\leq 3},
\end{equation}
and we will identify such a map with the collection of its structure coefficients $(\beta_{i,j,k})\in\mathbb C^{18}$. Following R\"ohrl \cite{Rohrl} we introduce some terminology here which is adapted from the theory of nonassociative algebras; see also subsection \ref{rohrlsec} below.
An {\em idempotent} of $p$ is a $v\in\mathbb C^n$ such that $p(v)=v\not=0$;  and $w\not=0$ with $p(w)=0$ is called a {\em nilpotent}. It is known that generically (corresponding to a Zariski--open and dense subset of coefficient space) a homogeneous quadratic vector field posseses no nilpotent (see e.g. R\"ohrl \cite{Rohrl}, Theorem 1), and that vector fields without a nilpotent have only finitely many idempotents (otherwise the variety in $\mathbb P^3(\C)$ defined by $p(x)-\xi\cdot x=0$ would have positive dimension and would intersect the hyperplane given by $\xi=0$). By Lemma \ref{l2}, at most seven idempotents exist. According to Theorem \ref{throhrl} below, generically there exists a basis of idempotents, and one may infer from its proof that generically there are exactly seven idempotents. We use this observation to discuss a special class of homogeneous quadratic vector fields.
\begin{definition}\label{disdef} We call the homogeneous quadratic vector field $p$ in $\mathbb C^3$ {\em distinguished} if
\begin{enumerate}[(i)]
\item $p$ admits the standard basis elements $e_1,\,e_2,\,e_3$ as idempotents;
\item there are three further idempotents $v_1,\,v_2,\,v_3$ determined by
\begin{equation}\label{distrel}
v_i=\gamma_{i,1}e_1+\gamma_{i,2}e_2+\gamma_{i,3}e_3; \quad 1\leq i\leq 3;
\end{equation}
with complex coefficients $\gamma_{i,j}$;
\item the matrix
\[
A:=\begin{pmatrix}\gamma_{11}\gamma_{12}&\gamma_{12}\gamma_{13}&\gamma_{13}\gamma_{11}\\
\gamma_{21}\gamma_{22}&\gamma_{22}\gamma_{23}&\gamma_{23}\gamma_{21}\\
\gamma_{31}\gamma_{32}&\gamma_{32}\gamma_{33}&\gamma_{33}\gamma_{31}\end{pmatrix}
\]
is invertible.
\end{enumerate}
\end{definition}
From these data $p$ can be reconstructed, since $p$ corresponds to a symmetric bilinear map
\[
\widehat p:\C^3\times\C^3\to\C^3, \quad(u,v)\mapsto\widehat p(u,v):=\frac12\left(p(u+v)-p(u)-p(v)\right)
\]
with $\widehat p(u,u)=p(u)$ for all $u$. Thus $p$ is uniquely determined by the $\widehat p(e_i,e_j)$, and these may be obtained from the relations
\[
\widehat p(\gamma_{i,1}e_1+\gamma_{i,2}e_2+\gamma_{i,3}e_3,\gamma_{i,1}e_1+\gamma_{i,2}e_2+\gamma_{i,3}e_3)=\gamma_{i,1}e_1+\gamma_{i,2}e_2+\gamma_{i,3}e_3
\]
and bilinearity. The neccesary calculations for this and for further steps require a computer algebra system (we use {\sc Maple} in the present paper).
 As it turns out, stipulating the idempotents in \eqref{distrel} defines a unique homogeneous quadratic map whenever $\det A$ does not vanish; see subsection \ref{detailsapp} below. In coordinates one finds an expression
\begin{equation}\label{sysquadr}
p(x)=\begin{pmatrix}x_1^2+\theta_1x_1x_2+\theta_2x_2x_3+\theta_3x_3x_1\\
x_2^2+\theta_4x_1x_2+\theta_5x_2x_3+\theta_6x_3x_1\\
x_3^2+\theta_7x_1x_2+\theta_8x_2x_3+\theta_9x_3x_1\end{pmatrix},
\end{equation}
which corresponds to a nine-dimensional affine subspace $Y$ of coefficient space $\mathbb C^{18}$, and the $\theta_k$ are rational functions in the $\gamma_{ij}$; see subsection \ref{detailsapp} below for the explicit form.
\begin{lemma}\label{denselem}
The image of the map
\[
\Gamma: \C^9\to Y, \quad \left(\gamma_{ij}\right)\mapsto\left(\theta_1\left((\gamma_{ij})\right),\ldots,\theta_9\left((\gamma_{ij})\right)\right)
\]
contains a Zariski--open subset of $Y$.
\end{lemma}
\begin{proof}
It is sufficient to show that the Jacobian of this map is invertible at some $\left(\widehat\gamma_{ij}\right)$, and this can be verified by direct calculation using {\sc Maple}; see subsection \ref{detailsapp}.
\end{proof}
From vector fields with the standard basis elements as idempotents, thus with coefficients in $Y$, one obviously obtains all vector fields admitting a basis of idempotents by linear coordinate transformations $T\in GL_3(\C)$, sending $p$ to $T^{-1}\circ p\circ T$. To summarize, we have the first statement of the following proposition; the proof of the second statement (which is computationally involved) will be outlined in subsection \ref{detailsapp}.
\begin{proposition}\label{denseprop}
\begin{enumerate}[(a)]
\item The set of coordinate transformations of the distinguished homogeneous quadratic vector fields (seen as a subset of coefficient space) contains a Zariski--open set.
\item All distinguished vector fields have precisely seven idempotents, with the coordinates of the seventh idempotent being rational in the $\gamma_{ij}$.
\end{enumerate}
\end{proposition}
If $v$ is any idempotent of $p$ then $2$ is an eigenvalue of the Jacobian $Dp(v)$, with eigenvector $v$. Denote the remaining ones by $\lambda_1$ and $\lambda_2$, noting that these lie in a degree two extension of the rational function field $\C\left((\gamma_{ij})_{i,j}\right)$, and explicit expressions for them can be determined using computer algebra. The eigenvalues for the linearization of the Poincar\'e transform are then $-1, \lambda_1-1$ and $\lambda_2-1$, according to Lemma \ref{l2}. With this in hand, one can show that quadratic vector fields with property $E$ are indeed generic.
\begin{proposition}\label{genprop1}
\begin{enumerate}[(a)]
\item Whenever the $\gamma_{ij}$ are algebraically independent over the rational numbers $\mathbb Q$ then the distinguished homogeneous quadratic vector field constructed with these parameters satisfies property E.
\item The homogeneous quadratic vector fields \eqref{quadcoeff} which satisfy property E correspond to the complement of a Lebesgue measure zero subset of parameter space.
\end{enumerate}
\end{proposition}
\begin{proof}
For statement (a) one uses computer algebra, by inspecting the eigenvalues and verifying that they are linearly independent over the rationals. It is sufficient to do so for a specialization, assigning rational values to some of the parameters and leaving only three algebraically independent ones. (Moreover one may work directly with the eigenvalues of the Jacobians at idempotents due to Lemma \ref{l2}(b); computing the Poincar\'e transform is not necessary.) See subsection \ref{detailsapp} below.

To prove statement (b), recall that the set of parameters which are algebraically dependent over $\mathbb Q$ is of Lebesgue measure zero in $\mathbb C^9$, and this property transfers to the corresponding subset of $Y$ given by the image of the map $\Gamma$, which is generically locally invertible. Using coordinate transformations as the last step, the claim is proven.
\end{proof}
This statement is not yet quite satisfactory, since one knows that there is an open and dense subset of coefficient space so that \eqref{quadcoeff} admits no semi--invariant at all; see \.Zo{\l}\c{a}dek \cite{Zoladek}. Therefore we next ascertain the existence of vector fields which have property E and admit nontrivial semi--invariants. This is taken care of by the next result.
\begin{proposition}\label{genprop2}
In Definition \ref{disdef}, let $\gamma_{13}=\gamma_{21}=\gamma_{32}=0 $, with the remaining $\gamma_{ij}$ algebraically independent over the rational numbers $\mathbb Q$. Then $x_i,\, 1\leq i\leq 3$, is a semi--invariant of the distinguished vector field,  and this vector field satisfies property E.
\end{proposition}
\begin{proof}
All claims are again proven by inspection of computer algebra calculations; see subsection \ref{detailsapp}.
\end{proof}
Finally, we exhibit an example which shows the existence of distinguished quadratic vector fields with algebraic coefficients $\gamma_{ij}$.
\begin{example}\label{algex}{\em
With the algebraic coefficients
$$
  \gamma_{{11}}=\sqrt {2},\ \gamma_{{12}}=\sqrt {3},\ \gamma_{{13}}
=0,\ \gamma_{{21}}=0,\ \gamma_{{22}}=\sqrt {3},\ \gamma_{{23}}=\sqrt {5},
\ \gamma_{{31}}=\sqrt {2},\ \gamma_{{32}}=0,\ \gamma_{{33}}=\sqrt {5},
$$
the distinguished system has components
\begin{equation}\label{ex1}
\begin{array}{cl}
p_1(x)&={{ x_1}}^{2}-\dfrac{ \left( 10\,\sqrt {2}\sqrt {3
}-10\,\sqrt {3} \right)}{30} { x_1}\,{ x_2}\,-\,\dfrac{ \left( 6\,\sqrt {2
}\sqrt {5}-6\,\sqrt {5} \right)}{30}{ x_1}\,{ x_3},\\

p_2(x)&={{ x_2}}^{2}-\dfrac{ \left( 15\,\sqrt {2}\sqrt {3
}-15\,\sqrt {2} \right)}{30}{ x_1}\,{ x_2} -\dfrac{\, \left( 6\,\sqrt {3
}\sqrt {5}-6\,\sqrt {5} \right)}{30}{ x_2}\,{ x_3},\\

p_3(x)&={{ x_3}}^{2}-\dfrac{\left( \sqrt {5}-1 \right) \sqrt {2}}{2}{ x_3}\,{ x_1}-
\dfrac{ \left( \sqrt {5}-1 \right) \sqrt {3}}{3}{ x_3}\,{ x_2}.
\end{array}
\end{equation}
Note that this system admits the invariant surfaces given by $x_1=0$, $x_2=0$,  resp. $x_3=0$. We will show in subsection \ref{detailsapp} that property E is satisfied by exhibiting the eigenvalues of the Jacobians for all idempotents. To verify linear independence of these eigenvalues over $\mathbb Q$ by inspection, recall that $\sqrt{2},\,\sqrt{3}$ and $\sqrt{5}$ generate a field extension of degree $8$ over $\mathbb Q$.}
\end{example}
%%%%%%%%%%%%%%%%%%%%%%%%%%%%%%%%%%%%%%%%%%%%
\section{Appendix}
\subsection{A proof of Lemma \ref{polyloc}}\label{polylocproof}
We give here an elementary proof, using only basic properties of polynomial and power series rings.
\begin{enumerate}[1.]
\item Let $m>0$ and
\[
\phi(x_1,\ldots,x_n)= x_n^m+\alpha_1(x_1,\ldots,x_{n-1})x_n^{m-1}+\cdots +\alpha_m(x_1,\ldots,x_{n-1})\in\C[x_1,\ldots,x_n]
\]
 be a polynomial with $\phi(0,\ldots,0)=0$. Then no prime factor of $\phi$ in the formal power series ring $\C[[x_1,\ldots,x_n]]$ lies in $\C[[x_1,\ldots,x_{n-1}]]$.

Indeed, a factorization $\phi(x_1,\ldots,x_n)=\sigma(x_1,\ldots,x_n)\cdot \tau(x_1,\ldots,x_{n-1})$ with non-invertible $\tau$, hence $\tau(0,\ldots,0)=0$, would yield the contradiction $\phi(0,\ldots,0,x_n)=0$.
\item We now show: If $\phi$ and $\gamma$ are relatively prime polynomials in $n$ variables, with $\phi(0)=\gamma(0)=0$, then $\phi$ and $\gamma$ remain relatively prime in the power series ring $\C[[x_1,\ldots,x_n]]$. (This implies statement (b) of Lemma \ref{polyloc}.)

To show this we may assume that $\phi$ is as above and likewise
\[
\gamma=x_n^k+\beta_1(x_1,\ldots,x_{n-1})x_n^{k-1}+\cdots +\beta_k(x_1,\ldots,x_{n-1})
\]
with some $k>0$. (One may achieve such a form by applying a linear transformation and multiplication by nonzero constants, which does not affect the statement.) Since $\phi$ and $\gamma$ are relatively prime polynomials, their resultant $\rho$ with respect to $x_n$ lies in $\C[x_1,\ldots,x_{n-1}]$, and there exist polynomials $\mu$ and $\nu$ such that
\[
\mu\cdot\phi+\nu\cdot\gamma=\rho;
\]
see Cox et al. \cite{CLOS}, Chapter 3, \S 6, Proposition 1. Assuming that a common prime factor of $\phi$ and $\gamma$ exists in $\C[[x_1,\ldots,x_{n}]]$, this factor also divides $\rho$ and therefore lies in $\C[[x_1,\ldots,x_{n-1}]]$; a contradiction.
\item To prove statement (a) of Lemma \ref{polyloc}, assume that $\phi$ has the form above and is irreducible, then apply the previous argument to $\phi$ and $\partial\phi/\partial x_n$ to see that $\phi$ cannot admit multiple prime factors in $\C[[x_1,\ldots,x_{n}]]$.
\end{enumerate}
\subsection{A result of H.~R\"ohrl.}\label{rohrlsec}
In \cite{Rohrlidpo}, R\"ohrl stated a theorem on criteria for (in his terminology) $m$-ary algebras to admit a basis of idempotents. While the theorem as stated is incorrect, the weaker statement that generically an $m$-ary algebra admits a basis of idempotents is correct, and R\"ohrl's arguments can be modified to prove it. Here we give a proof that in part takes  a different approach.

Consider a homogeneous polynomial map
\begin{equation}
p:\,\mathbb C^n\to\mathbb C^n; \quad p(x)=\left(\sum_{i_1,\ldots,i_n}\eta_{i_1,\ldots,i_n,k}x_1^{i_1}\cdots x_n^{i_n}\right)_{1\leq k\leq n}
\end{equation}
of degree $m>1$.  As noted above, an {\em idempotent} of $p$ is a $v\in\mathbb C^n$ such that $p(v)=v\not=0$. Note that any $w\not=0$ that satisfies $p(w)=\beta\cdot w$ for some $\beta\not=0$ is a scalar multiple of an idempotent. (If $w\not=0$ and $p(w)=0$ then $w$ is a {\em nilpotent}.)

In order to state the result properly, we need further terminology.
Denoting the collection of structure coefficients $\eta_{i_1,\ldots,i_n,k}$ by $y\in\mathbb C^N$, we write
\begin{equation}
p(x)=Q(y,x).
\end{equation}

\begin{theorem}\label{throhrl}
The set of all $y$ such that $\mathbb C^n$ admits a basis of idempotents of $p=Q(y,\cdot)$ contains a Zariski--open (and dense) subset of parameter space $\mathbb C^N$.
\end{theorem}
\begin{proof}
We set $F(y,x):=Q(y,x)-x$.  Then, for a fixed parameter $y_0$, an idempotent of $Q(y_0,\,\cdot)$ is a nonzero solution of the equation
\[
Q(y_0,x)-x=0,\text{  equivalently  } F(y_0,x)=0.
\]
\begin{enumerate}[A.]
\item We first show the existence of a norm--open set in parameter space that satisfies the desired property.
\begin{enumerate}[1.]
\item We use the Implicit Function Theorem: Assume that $x_0$ is a nonzero solution of $F(y_0,x)=0$, and that the partial derivative
\[
D_xF(y_0,x_0)=D_xQ(y_0,x_0)-I
\]
is invertible. Then there exists an analytic function $h$ from some neighborhood $U$ of $y_0$ to $\mathbb C^n$ such that
\[
h(y_0)=x_0\text{   and   }Q(y,h(y))-h(y)=0,\text{  for all   }y\in U.
\]
Thus, for any small change of the structure coefficients there is a nearby idempotent. (The condition on the derivative, incidentally, means that the idempotent $x_0$ is of multiplicity one.)
\item The same argument can be applied to a basis of idempotents: Assume there is a parameter $y_0$ such that the equation $F(y_0,x)=0$ admits $n$ linearly independent solutions $x^{(1)}_0,\ldots, x^{(n)}_0$, and that the partial derivative with respect to $x$ at $(y,x_0^{(i)})$ is invertible. Then consider the map
\[
G: \mathbb C^N\times\left(\mathbb C^n\right)^n\to \left(\mathbb C^n\right)^n,\quad(y,x^{(1)},\ldots,x^{(n)})\mapsto(F(y,x^{(1)}),\ldots,F(y,x^{(n)})).
\]
Here we have a solution $(y_0,x_0^{(1)},\ldots,x_0^{(n)})$ of $G=0$ and the matrix representing the partial derivative with respect to $(x^{(1)},\ldots,x^{(n)})$ is block diagonal with invertible blocks, hence invertible. Therefore there exist a neighborhood $U$ of $y_0$ and analytic maps
\[
h^{(i)}:\,U\to\mathbb C^n\text{   with   }h^{(i)}(y_0)=x^{(i)}_0\text{  and  }Q(y,h^{(i)}(y))-h^{(i)}(y)=0
\]
for $ y\in U$, $1\leq i\leq n$. Moreover
\[
\det\left(h^{(1)}(y),\ldots,h^{(n)}(y)\right)
\]
does not vanish at $y_0$, hence in some open neighborhood $\widetilde U\subseteq U$ of $y_0$. In other words, for all $y\in\widetilde U$ the corresponding $p=Q(y,\cdot)$ admits a basis of idempotents.
\item To finish the argument, one needs to exhibit one $p$ (thus some $y_0$) for which a basis of idempotents exists and the respective partial derivatives are invertible: Take
\[
p^*(x)=\begin{pmatrix}x_1^m\\ \vdots\\x_n^m\end{pmatrix}=:Q(y^*,x),
\]
and the standard basis $e_1,\ldots,e_n$.
\end{enumerate}

\item  Now we show the existence of a Zariski--open set in parameter space which has the desired properties.
For $1\leq k<\ell\leq n$  define
\[
\Delta_{k,\ell}(u,v):=u_kv_\ell-u_\ell v_k;\quad u,v\in\mathbb C^n.
\]
Consider the morphism
\[
\begin{array}{rl}
H:&\mathbb C^N\times\left(\mathbb C^n\right)^n\to\left(\mathbb C^{n(n-1)/2}\right)^n\times \mathbb C,\\
  &\left(y,x^{(1)},\ldots,x^{(n)}\right)\mapsto\begin{pmatrix} \Delta_{k,\ell}(Q(y,x^{(i)}),x^{(i)})_{i,k,\ell}\\ \det(x^{(1)},\ldots,x^{(n)})-1\end{pmatrix}.
\end{array}
\]
Then $H=0$ defines an algebraic subvariety $Z$ of $\mathbb C^N\times\left(\mathbb C^n\right)^n$, which is nonempty by the first part of the proof. Any zero of $H$ corresponds to some homogeneous $p$ which admits a basis $v_1,\ldots, v_n$ of $\mathbb C^n$ with $p(v_i)\in \mathbb C v_i$.

The image $\widetilde Z$ of $Z$ under the projection
\[
\pi: \mathbb C^N\times\left(\mathbb C^n\right)^n\to\mathbb C^N
\]
onto the first component contains a Zariski--open subset of its closure, by a standard theorem on morphisms of algebraic varieties (see e.g. Shafarevich \cite{Sha}, Chapter 1, Theorem 1.14). But due to the first part of the proof, $\widetilde Z$ also contains a norm--open set $U$, whose Zariski closure is all of $\mathbb C^N$. Therefore $\widetilde Z$ contains a nonempty Zariski--open subset of $\mathbb C^N$.

Finally, those $p$ possessing a nilpotent correspond to a proper Zariski--closed subset of parameter space, thanks to a resultant argument; see e.g. R\"ohrl \cite{Rohrl}, Theorem 1. Taking the complement of this closed subset, we find that all $p(v_i)\in\mathbb C v_i\setminus\{0\}$, whence some multiple of $v_i$ is an idempotent, and the proof is finished.
\end{enumerate}
\end{proof}
%%%%%%%%%%%%%%%%%%%%%%%%%%%%%%%%%%%%%%%%%%%
\subsection{Construction of quadratic vector fields in dimension three: Some details.}\label{detailsapp}
Here we describe some details of the construction of distinguished quadratic vector fields, and outline some arguments. We also supply, in the additional material, some {\sc Maple} worksheets containing the calculations.
\subsubsection{The explicit expression for system \eqref{sysquadr}}
The quadratic homogeneous vector field \eqref{sysquadr} having the prescribed six idempotents $e_1,e_2,e_3, v_1,v_2,v_3$ with $v_i=\gamma_{i1}e_1+\gamma_{i2}e_2+\gamma_{i3}e_3$ and $i=1,2,3$ is given by
$$
\begin{array}{rl}
p_1=&{{ x_1}}^{2}+\left[ \,{ x_1}\,{ x_2}\, \left( -{{ \gamma_{11}}}^{2}
{ \gamma_{21}}\,{ \gamma_{23}}\,{ \gamma_{32}}\,{ \gamma_{33}}
+{{ \gamma_{11}}}^{2}{ \gamma_{22}}\,{ \gamma_{23}}\,{ \gamma_{31}}\,{ \gamma_{33}}+{ \gamma_{11}}\,{ \gamma_{13}}\,{{ \gamma_{21}}}^{2}{
 \gamma_{32}}\,{ \gamma_{33}}-{ \gamma_{11}}\,{ \gamma_{13}}\,{ \gamma_{22}}\,{ \gamma_{23}}\,{{ \gamma_{31}}}^{2}\right.\right.\\

&-{ \gamma_{12}}\,{ \gamma_{13}}\,{{ \gamma_{21}}}^{2}{ \gamma_{31}}\,{ \gamma_{33}}
+{ \gamma_{12}}\,{ \gamma_{13}}\,{ \gamma_{21}}\,{ \gamma_{23}}\,{{ \gamma_{31}}}^{2}
-{ \gamma_{11}}\,{ \gamma_{13}}\,{ \gamma_{21}}\,{ \gamma_{32}}\,{ \gamma_{33}}
+{ \gamma_{11}}\,{ \gamma_{13}}\,{ \gamma_{22}}\,{ \gamma_{23}}\,{ \gamma_{31}}\\

&\left.+{ \gamma_{11}}\,{ \gamma_{21}}\,{ \gamma_{23}}\,{ \gamma_{32}}
\,{ \gamma_{33}}-{ \gamma_{11}}\,{ \gamma_{22}}\,{ \gamma_{23}}\,{ \gamma_{31}}\,{ \gamma_{33}}-{
 \gamma_{12}}\,{ \gamma_{13}}\,{ \gamma_{21}}\,{ \gamma_{23}}\,{ \gamma_{31}}+{ \gamma_{12}}\,{
\gamma_{13}}\,{ \gamma_{21}}\,{ \gamma_{31}}\,{ \gamma_{33}} \right)\\

& -\,{ x_2}\,{ x_3}\,{ \gamma_{11}}\,{ \gamma_{21}}\,{ \gamma_{31}}\,
  \left( { \gamma_{11}}\,{ \gamma_{22}}\,{
\gamma_{33}}-{ \gamma_{11}}\,{ \gamma_{23}}\,{ \gamma_{32}}-{ \gamma_{12}}\,{ \gamma_{21}}\,{ \gamma_{33}}+{
 \gamma_{12}}\,{ \gamma_{23}}\,{ \gamma_{31}}+{ \gamma_{13}}\,{ \gamma_{21}}\,{ \gamma_{32}}-
{ \gamma_{13}}\,{ \gamma_{22}}\,{ \gamma_{31}}\right.\\

&\left.-{ \gamma_{12}}\,{ \gamma_{23}}+{ \gamma_{12}}\,{ \gamma_{33}}
+{ \gamma_{13}}\,{ \gamma_{22}}-{ \gamma_{13}}\,{ \gamma_{32}}-{ \gamma_{22}}\,{ \gamma_{33}}+{ \gamma_{23}
}\,{ \gamma_{32}} \right) \,\\

&+\,{ x_1}\,{ x_3} \left( {{ \gamma_{11}}}^{2}{
 \gamma_{21}}\,{ \gamma_{22}}\,{ \gamma_{32}}\,{ \gamma_{33}}-{{ \gamma_{11}}}^{2}{ \gamma_{22}}\,{
 \gamma_{23}}\,{ \gamma_{31}}\,{ \gamma_{32}}-{ \gamma_{11}}\,{ \gamma_{12}}\,{{ \gamma_{21}}}^{2}{
 \gamma_{32}}\,{ \gamma_{33}}+{ \gamma_{11}}\,{ \gamma_{12}}\,{ \gamma_{22}}\,{ \gamma_{23}}\,{{
\gamma_{31}}}^{2}\right.\\

&+{ \gamma_{12}}\,{ \gamma_{13}}\,{{ \gamma_{21}}}^{2}{ \gamma_{31}}\,{ \gamma_{32}}-{
 \gamma_{12}}\,{ \gamma_{13}}\,{ \gamma_{21}}\,{ \gamma_{22}}\,{{ \gamma_{31}}}^{2}+{ \gamma_{11}}\,
{ \gamma_{12}}\,{ \gamma_{21}}\,{ \gamma_{32}}\,{ \gamma_{33}}-{ \gamma_{11}}\,{ \gamma_{12}}\,{
\gamma_{22}}\,{ \gamma_{23}}\,{ \gamma_{31}}\\

&\left.\left.-{ \gamma_{11}}\,{ \gamma_{21}}\,{ \gamma_{22}}\,{ \gamma_{32}}
\,{ \gamma_{33}}+{ \gamma_{11}}\,{ \gamma_{22}}\,{ \gamma_{23}}\,{ \gamma_{31}}\,{ \gamma_{32}}+{
 \gamma_{12}}\,{ \gamma_{13}}\,{ \gamma_{21}}\,{ \gamma_{22}}\,{ \gamma_{31}}-{ \gamma_{12}}\,{
\gamma_{13}}\,{ \gamma_{21}}\,{ \gamma_{31}}\,{ \gamma_{32}} \right)\right]/d
\end{array}
$$

$$
\begin{array}{rl}
p_2&={{ x_2}}^{2}+\left[ \,{ x_1}\,{ x_2}\, \left( { \gamma_{11}}\,{
\gamma_{13}}\,{{ \gamma_{22}}}^{2}{ \gamma_{32}}\,{ \gamma_{33}}-{ \gamma_{11}}\,{ \gamma_{13}}\,{
\gamma_{22}}\,{ \gamma_{23}}\,{{ \gamma_{32}}}^{2}-{{ \gamma_{12}}}^{2}{ \gamma_{21}}\,{ \gamma_{23}}\,
{ \gamma_{32}}\,{ \gamma_{33}}\right.\right.\\

&+{{ \gamma_{12}}}^{2}{ \gamma_{22}}\,{ \gamma_{23}}\,{ \gamma_{31}}\,{
 \gamma_{33}}+{ \gamma_{12}}\,{ \gamma_{13}}\,{ \gamma_{21}}\,{ \gamma_{23}}\,{{ \gamma_{32}}}^{2}-{
 \gamma_{12}}\,{ \gamma_{13}}\,{{ \gamma_{22}}}^{2}{ \gamma_{31}}\,{ \gamma_{33}}+{ \gamma_{11}}\,{
 \gamma_{13}}\,{ \gamma_{22}}\,{ \gamma_{23}}\,{ \gamma_{32}}\\

&-{ \gamma_{11}}\,{ \gamma_{13}}\,{
\gamma_{22}}\,{ \gamma_{32}}\,{ \gamma_{33}}-{ \gamma_{12}}\,{ \gamma_{13}}\,{ \gamma_{21}}\,{ \gamma_{23}}
\,{ \gamma_{32}}+{ \gamma_{12}}\,{ \gamma_{13}}\,{ \gamma_{22}}\,{ \gamma_{31}}\,{ \gamma_{33}}+{
 \gamma_{12}}\,{ \gamma_{21}}\,{ \gamma_{23}}\,{ \gamma_{32}}\,{ \gamma_{33}}\\

&\left.-{ \gamma_{12}}\,{
\gamma_{22}}\,{ \gamma_{23}}\,{ \gamma_{31}}\,{ \gamma_{33}} \right) \\

&+\,{ x_2}\,{ x_3}\,
 \left( -{ \gamma_{11}}\,{ \gamma_{12}}\,{ \gamma_{21}}\,{ \gamma_{23}}\,{{ \gamma_{32}}}^{2}+{
 \gamma_{11}}\,{ \gamma_{12}}\,{{ \gamma_{22}}}^{2}{ \gamma_{31}}\,{ \gamma_{33}}+{ \gamma_{11}}\,{
 \gamma_{13}}\,{ \gamma_{21}}\,{ \gamma_{22}}\,{{ \gamma_{32}}}^{2}-{ \gamma_{11}}\,{ \gamma_{13}}\,
{{ \gamma_{22}}}^{2}{ \gamma_{31}}\,{ \gamma_{32}}\right.\\

&-{{ \gamma_{12}}}^{2}{ \gamma_{21}}\,{ \gamma_{22}
}\,{ \gamma_{31}}\,{ \gamma_{33}}+{{ \gamma_{12}}}^{2}{ \gamma_{21}}\,{ \gamma_{23}}\,{ \gamma_{31}}
\,{ \gamma_{32}}+{ \gamma_{11}}\,{ \gamma_{12}}\,{ \gamma_{21}}\,{ \gamma_{23}}\,{ \gamma_{32}}-{
 \gamma_{11}}\,{ \gamma_{12}}\,{ \gamma_{22}}\,{ \gamma_{31}}\,{ \gamma_{33}}\\

&\left.
-{ \gamma_{11}}\,{
\gamma_{13}}\,{ \gamma_{21}}\,{ \gamma_{22}}\,{ \gamma_{32}}
+{ \gamma_{11}}\,{ \gamma_{13}}\,{ \gamma_{22}}
\,{ \gamma_{31}}\,{ \gamma_{32}}+{ \gamma_{12}}\,{ \gamma_{21}}\,{ \gamma_{22}}\,{ \gamma_{31}}\,{
 \gamma_{33}}-{ \gamma_{12}}\,{ \gamma_{21}}\,{ \gamma_{23}}\,{ \gamma_{31}}\,{ \gamma_{32}}
 \right)\\

 &-\,{ x_1}\,{ x_3}\,{ \gamma_{12}}\,{ \gamma_{22}}\,{ \gamma_{32}}\,
 \left( { \gamma_{11}}\,{ \gamma_{22}}\,{ \gamma_{33}}-{ \gamma_{11}}\,{ \gamma_{23}}\,{ \gamma_{32}
}-{ \gamma_{12}}\,{ \gamma_{21}}\,{ \gamma_{33}}+{ \gamma_{12}}\,{ \gamma_{23}}\,{ \gamma_{31}}\right.\\

&\left.\left.
+{
 \gamma_{13}}\,{ \gamma_{21}}\,{ \gamma_{32}}-{ \gamma_{13}}\,{ \gamma_{22}}\,{ \gamma_{31}}+{
\gamma_{11}}\,{ \gamma_{23}}-{ \gamma_{11}}\,{ \gamma_{33}}-{ \gamma_{13}}\,{ \gamma_{21}}+{ \gamma_{13}}\,{
 \gamma_{31}}+{ \gamma_{33}}\,{ \gamma_{21}}-{ \gamma_{23}}\,{ \gamma_{31}} \right)\right]/d
\end{array}
$$

$$
\begin{array}{rl}
p_3=&{x_3}^2- \left[{ x_1}\,{ x_2}{ \gamma_{13}}\,{ \gamma_{23}}\,{ \gamma_{33}}\, \left( { \gamma_{11}}\,{ \gamma_{22}}\,{
\gamma_{33}}-{ \gamma_{11}}\,{ \gamma_{23}}\,{ \gamma_{32}}-{ \gamma_{12}}\,{ \gamma_{21}}\,{ \gamma_{33}}+{
 \gamma_{12}}\,{ \gamma_{23}}\,{ \gamma_{31}}+{ \gamma_{13}}\,{ \gamma_{21}}\,{ \gamma_{32}}-\right.\right.\\

&\left.
{
\gamma_{13}}\,{ \gamma_{22}}\,{ \gamma_{31}}-{ \gamma_{11}}\,{ \gamma_{22}}+{ \gamma_{32}}\,{ \gamma_{11}}+{
 \gamma_{12}}\,{ \gamma_{21}}-{ \gamma_{31}}\,{ \gamma_{12}}-{ \gamma_{21}}\,{ \gamma_{32}}+{ \gamma_{31}
}\,{ \gamma_{22}} \right)\\

&+{ x_1}\,{ x_3} \left( { \gamma_{11}}\,{ \gamma_{12}}\,
{ \gamma_{22}}\,{ \gamma_{23}}\,{{ \gamma_{33}}}^{2}-{ \gamma_{11}}\,{ \gamma_{12}}\,{{ \gamma_{23}}
}^{2}{ \gamma_{32}}\,{ \gamma_{33}}-{ \gamma_{12}}\,{ \gamma_{13}}\,{ \gamma_{21}}\,{ \gamma_{22}}\,
{{ \gamma_{33}}}^{2}+{ \gamma_{12}}\,{ \gamma_{13}}\,{{ \gamma_{23}}}^{2}{ \gamma_{31}}\,{
\gamma_{32}}\right.\\

&+{{ \gamma_{13}}}^{2}{ \gamma_{21}}\,{ \gamma_{22}}\,{ \gamma_{32}}\,{ \gamma_{33}}-{{
\gamma_{13}}}^{2}{ \gamma_{22}}\,{ \gamma_{23}}\,{ \gamma_{31}}\,{ \gamma_{32}}-{ \gamma_{11}}\,{
\gamma_{12}}\,{ \gamma_{22}}\,{ \gamma_{23}}\,{ \gamma_{33}}+{ \gamma_{11}}\,{ \gamma_{12}}\,{ \gamma_{23}}
\,{ \gamma_{32}}\,{ \gamma_{33}}\\

&\left.+{ \gamma_{12}}\,{ \gamma_{13}}\,{ \gamma_{21}}\,{ \gamma_{22}}\,{
 \gamma_{33}}-{ \gamma_{12}}\,{ \gamma_{13}}\,{ \gamma_{23}}\,{ \gamma_{31}}\,{ \gamma_{32}}-{
\gamma_{13}}\,{ \gamma_{21}}\,{ \gamma_{22}}\,{ \gamma_{32}}\,{ \gamma_{33}}+{ \gamma_{13}}\,{ \gamma_{22}}
\,{ \gamma_{23}}\,{ \gamma_{31}}\,{ \gamma_{32}} \right) \\

&- { x_2}
\,{ x_3}\left( {
 \gamma_{11}}\,{ \gamma_{12}}\,{ \gamma_{21}}\,{ \gamma_{23}}\,{{ \gamma_{33}}}^{2}-{ \gamma_{11}}\,
{ \gamma_{12}}\,{{ \gamma_{23}}}^{2}{ \gamma_{31}}\,{ \gamma_{33}}-{ \gamma_{11}}\,{ \gamma_{13}}\,{
 \gamma_{21}}\,{ \gamma_{22}}\,{{ \gamma_{33}}}^{2}+{ \gamma_{11}}\,{ \gamma_{13}}\,{{ \gamma_{23}}}
^{2}{ \gamma_{31}}\,{ \gamma_{32}}\right.\\

&+{{ \gamma_{13}}}^{2}{ \gamma_{21}}\,{ \gamma_{22}}\,{ \gamma_{31}
}\,{ \gamma_{33}}
-{{ \gamma_{13}}}^{2}{ \gamma_{21}}\,{ \gamma_{23}}\,{ \gamma_{31}}\,{ \gamma_{32}}
-{ \gamma_{11}}\,{ \gamma_{12}}\,{ \gamma_{21}}\,{ \gamma_{23}}\,{ \gamma_{33}}+{ \gamma_{11}}\,{
 \gamma_{12}}\,{ \gamma_{23}}\,{ \gamma_{31}}\,{ \gamma_{33}}
\\

&\left.\left.+{ \gamma_{11}}\,{ \gamma_{13}}\,{
\gamma_{21}}\,{ \gamma_{22}}\,{ \gamma_{33}}
-{ \gamma_{11}}\,{ \gamma_{13}}\,{ \gamma_{23}}\,{ \gamma_{31}}
\,{ \gamma_{32}}-{ \gamma_{13}}\,{ \gamma_{21}}\,{ \gamma_{22}}\,{ \gamma_{31}}\,{ \gamma_{33}}+{
 \gamma_{13}}\,{ \gamma_{21}}\,{ \gamma_{23}}\,{ \gamma_{31}}\,{ \gamma_{32}} \right) \right]/d
\end{array}
$$
with
$$
\begin{array}{cc}
d=\det(A)=&{ \gamma_{11}}\,{ \gamma_{12}}\,{ \gamma_{21}}\,{ \gamma_{23}}\,{ \gamma_{32}}\,{ \gamma_{33}}
-{ \gamma_{11}}\,{ \gamma_{12}}\,{ \gamma_{22}}\,{ \gamma_{23}}\,{ \gamma_{31}}\,{ \gamma_{33}}
-{ \gamma_{11}}\,{ \gamma_{13}}\,{ \gamma_{21}}\,{ \gamma_{22}}\,{ \gamma_{32}}\,{ \gamma_{33}}\\
&+{ \gamma_{11}}\,{ \gamma_{13}}\,{ \gamma_{22}}\,{ \gamma_{23}}\,{ \gamma_{31}}\,{ \gamma_{32}}
+{ \gamma_{12}}\,{ \gamma_{13}}\,{ \gamma_{21}}\,{ \gamma_{22}}\,{ \gamma_{31}}\,{ \gamma_{33}}
-{ \gamma_{12}}\,{ \gamma_{13}}\,{ \gamma_{21}}\,{ \gamma_{23}}\,{ \gamma_{31}}\,{ \gamma_{32}}.\\
\end{array}
$$
See the {\sc Maple} worksheet {\tt ExplicitSystem.}
\subsubsection{Concerning the proof of Proposition \ref{denseprop}}
For the values
\[
\gamma_{11}=-1,\,\gamma_{12}=3,\,\gamma_{13}=2,\,\gamma_{21}=1,\,\gamma_{22}=1,\,\gamma_{23}=-2,\,\gamma_{31}=0,\,\gamma_{32}=1,\,\gamma_{33}=-3
\]
the Jacobian of $\Gamma$ is invertible; see {\sc Maple} worksheet {\tt JacobianInvertible}.
\subsubsection{Computing the seventh idempotent}
To calculate the coefficients of the seventh idempotent $v=s_1e_1+s_2e_2+s_3e_3$ in terms of the nine parameters $\gamma_{ij}$,
first rewrite the components of the vector field into the form
$$
\begin{array}{cl}
p_1&=x_1^2+A_1(x_2,x_3)x_1+A_2(x_2,x_3),\\
p_2&=B_1(x_2,x_3)x_1+B_2(x_2,x_3), \\
p_3&=C_1(x_2,x_3)x_1+C_2(x_2,x_3),\\
\end{array}
$$
with
$$
\begin{array}{ll}
A_1=a_{11}x_2+a_{12}x_3-1, &  A_2=a_{13}x_2x_3,\\
B_1=b_{11}x_2+b_{12}x_3, & B_2=x_2^2-x_2+b_{13}x_2x_3,\\
C_1=c_{11}x_2+c_{12}x_3, & C_2=x_3^2-x_3+c_{13}x_2x_3.
 \end{array}
$$
(This is a preliminary step to avoid huge expressions involving the $\gamma_{ij}$ in {\sc Maple} worksheets.)
The idempotents are the nonzero solutions of the equation $p(x)-x=0$ where $x=(x_1,x_2,x_3)$. So we obtain
$$
x_1^2+A_1 x_1+A_2 =0, \ B_1 x_1+B_2=0,\  C_1 x_1+C_2=0.
$$
Here we may substitute  $x_1=-B_2/B_1$ (or $x_1=-C_2/C_1$.) From the former one obtains two equations
$$
B_1C_2-C_1B_2=0, \ B_2^2-A_1B_1B_2+A_2B_1^2=0 \text{  in variables  } x_2,\,x_3.
$$
Compute the resultant of this system with respect to the variable $x_2$. This is a polynomial of degree 12 in the variable $x_3$, of the form
$$
R(x_3)=x_3^5(x_3-1)b_{12}^2(b_{11}b_{13}x_3-b_{12}x_3-b_{11})^2\widetilde{T}_4(x_3)
$$
with $\widetilde{T}_4$ a polynomial of degree 4 in the variable $x_3.$ By construction, the third entries of the idempotents are among the zeros of $R(x_3)$. Here $e_1,\,e_2$ and $e_3$ correspond to a simple zero $1$ and a double zero $0$, and one verifies that the linear factor $b_{11}b_{13}x_3-b_{12}x_3-b_{11}$ does not correspond to the third entry of any $v_i$. So, the third entries of $v_1,\,v_2$ and $v_3$ are roots of $\widetilde T$. One may compute the fourth root of this degree four polynomial as follows: First normalize (divide by the leading coefficient) to obtain the monic polynomial $T_4$. For $T_4$, the sum of the four roots is the negative coefficient of  $x_3^3$. Since we know the three roots $\gamma_{13}, \gamma_{23}, \gamma_{33}$ of $T_4$ we are able to calculate the fourth one, which is rational in the $\gamma_{ij}$, after re-substitution. We take this as a candidate for $s_3$.

Similarly we may consider the resultant with respect to the variable $x_3$, which is a polynomial of degree 12 in the variable $x_2$ and takes the form
$$
R_(x_2)=x_2^5(x_2-1)b_{12}(b_{11}b_{13}x_2-b_{12}x_2+b_{12})^2\widetilde{S}_4(x_2),
$$
and proceed as above, finding a candidate for $s_2$. Finally use $x_1=-B_2/B_1$ to obtain a candidate for $s_1$, and verify that this indeed yields an idempotent by direct calculation; see {\sc Maple} worksheet {\tt 7thidempotent}. It should be noted that the output is generally quite voluminous.

\subsubsection{On the proofs of Propositions \ref{genprop1} and \ref{genprop2}}
It is sufficient to prove the latter, which is done by inspection of the eigenvalues of the Jacobians of all idempotents; see {\sc Maple} worksheet {\tt TestPropertyE}.
\subsubsection{Details for Example \ref{algex}}
Here we provide some details for the homogeneous quadratic vector field given in \eqref{ex1}, with parameters
$$
  \gamma_{{11}}=\sqrt {2},\ \gamma_{{12}}=\sqrt {3},\ \gamma_{{13}}
=0,\ \gamma_{{21}}=0,\ \gamma_{{22}}=\sqrt {3},\ \gamma_{{23}}=\sqrt {5},
\ \gamma_{{31}}=\sqrt {2},\ \gamma_{{32}}=0,\ \gamma_{{33}}=\sqrt {5}.
$$
In this case the output of the computations is of moderate size, so we can reproduce it here; see {\sc Maple} worksheet {\tt Algebraicallydependent} for the calculations.

The system has the seven idempotents $e_1,e_2,e_3, v_1=(\sqrt {2},\sqrt {3},0), v_2=(0,\sqrt {3},\sqrt {5}), v_3=(\sqrt {2},0,\sqrt {5})$ and
$v=(s_1, s_2,s_3)$, with
$$
\begin{array}{cl}
s_1&=-{\dfrac {  A_1A_2 }{A_3A_4  }},\\

s_2&={\dfrac { \left(  \left( -243\,\sqrt {5}+480 \right) \sqrt {3}+369\,
\sqrt {5}-900 \right) \sqrt {2}+ \left( 282\,\sqrt {5}-720 \right)
\sqrt {3}-504\,\sqrt {5}+1020}{ \left(  \left( 440\,\sqrt {5}-1000
 \right) \sqrt {3}-801\,\sqrt {5}+1785 \right) \sqrt {2}+ \left( -673
\,\sqrt {5}+1495 \right) \sqrt {3}+1191\,\sqrt {5}-2685}},\\

s_3&={\dfrac { \left(  \left( -345\,\sqrt {5}+695 \right) \sqrt {3}+495\,
\sqrt {5}-1245 \right) \sqrt {2}+ \left( 380\,\sqrt {5}-950 \right)
\sqrt {3}-780\,\sqrt {5}+1530}{ \left(  \left( 440\,\sqrt {5}-1000
 \right) \sqrt {3}-801\,\sqrt {5}+1785 \right) \sqrt {2}+ \left( -673
\,\sqrt {5}+1495 \right) \sqrt {3}+1191\,\sqrt {5}-2685}},
\end{array}
$$
and the abbreviations
$$
\begin{array}{cl}
A_1&= \left( 362\,\sqrt {2}\sqrt {3}\sqrt {5}-573\,\sqrt {2}\sqrt
{5}-545\,\sqrt {3}\sqrt {5}-780\,\sqrt {2}\sqrt {3}+969\,\sqrt {5}+
1335\,\sqrt {2}+1265\,\sqrt {3}-2085 \right),\\

A_2&=\left( 69\,\sqrt {2}
\sqrt {3}\sqrt {5}-99\,\sqrt {2}\sqrt {5}-76\,\sqrt {3}\sqrt {5}-139\,
\sqrt {2}\sqrt {3}+156\,\sqrt {5}+249\,\sqrt {2}+190\,\sqrt {3}-306
 \right),\\

A_3&= \left( 133\,\sqrt {2}\sqrt {3}\sqrt {5}-231\,\sqrt {2}
\sqrt {5}-208\,\sqrt {3}\sqrt {5}-285\,\sqrt {2}\sqrt {3}+348\,\sqrt {
5}+543\,\sqrt {2}+484\,\sqrt {3}-744 \right),\\

A_4&= \left( 440\,\sqrt {2}
\sqrt {3}\sqrt {5}-801\,\sqrt {2}\sqrt {5}-673\,\sqrt {3}\sqrt {5}-
1000\,\sqrt {2}\sqrt {3}+1191\,\sqrt {5}+1785\,\sqrt {2}+1495\,\sqrt {
3}-2685 \right).
\end{array}
$$

The eigenvalues of the Jacobian matrix at the idempotents $e_1,e_2$ and $e_3$ are respectively
$$
\begin{array}{c}
\left( \begin {array}{ccc} -\left( \sqrt {5}-1 \right) \sqrt{2}/2, \
 -\sqrt {2} \left( \sqrt {3}-1 \right)/2, \
2\end {array} \right),\\

\left( \begin {array}{ccc} -\, \left( \sqrt {5}-1 \right) \sqrt {3}/3,
\ -\,\sqrt {3} \left( \sqrt {2}-1 \right)/3,
\ 2\end {array} \right),\\

\left( \begin {array}{ccc} -\,\sqrt {5} \left( \sqrt {3}-1 \right)/5,\
-\,\sqrt {5} \left( \sqrt {2}-1 \right)/5, \
2\end {array} \right).
\end{array}
$$

The eigenvalues of the Jacobian matrix at the idempotents $v_1,v_2$ and $v_3$ are respectively
$$
\begin{array}{ccc}
 \left( \begin {array}{ccc} 2, \ \sqrt {2}+\sqrt {3}, \ -2\,\sqrt {5}+2\end {array} \right), &

 \left( \begin {array}{ccc} 2,\  -2\,\sqrt {2}+2,
\ \sqrt {5}+\sqrt {3}\end {array} \right), &

 \left( \begin {array}{ccc} 2,\ -2\,\sqrt {3}+2,
\ \sqrt {5}+\sqrt {2}\end {array} \right).
\end{array}
$$

The Jacobian matrix at the last idempotent $v=(s_1,s_2,s_3)$ has $2$ as an eigenvalue and the other two are
$$
\begin{array}{rcl}
\lambda_{\pm}&=&{\dfrac { \left(
 \left( 794929\,\sqrt {2}+762999 \right) \sqrt {3}+880620\,\sqrt {2}+
1744545 \right) \sqrt {5}}{4061514}}
+{\dfrac{ \left( 1796931\,\sqrt {2
}+3382333 \right) \sqrt {3}}{4061514}}\\
& &\pm\dfrac{\sqrt A}{4061514}
+{\dfrac {776393\,
\sqrt {2}}{676919}}+{\dfrac{3211015}{1353838}},
\end{array}
$$
with
$$
\begin{array}{l}
A= \left(  \left(
-6061791842292\,\sqrt {5}+20403579754296 \right) \sqrt {3}-
10627847112816\,\sqrt {5}+45521293739166 \right) \sqrt {2}\\
+ \left( -8804787537402\,\sqrt {5}+29950278886104 \right) \sqrt {3}-
15500011528278\,\sqrt {5}+66711726928548.
\end{array}
$$
In each case, inspection shows that the eigenvalues are linearly independent over the rational number field $\mathbb Q$.

\bigskip
\noindent{\bf Acknowledgements.} {NK acknowledges support by the DFG Research Training Group GRK 1632 ``Experimental and Constructive Algebra''.  JL is partially supported by the Ag\`encia de Gesti\'o
d'Ajuts Universitaris i de Recerca grant 2017SGR1617, and the H2020
European Research Council grant MSCA-RISE-2017-777911. JL and CP are also supported by the Ministerio de Ciencia, Innovaci\'on y
Universidades, Agencia Estatal de Investigaci\'on grant MTM2016-77278-P
(FEDER). CP is additionally partially supported by the Catalan Grant 2017SGR1049 and the Spanish MINECO-FEDER Grants MTM2015-65715-P and PGC2018-098676-B-I00/AEI/FEDER/UE.
Finally, CP and SW thank the CRM at Universitat Aut\'onoma de Barcelona for its hospitality during a research visit in May and June 2019.}

\end{document}